\documentclass[11pt]{articlefederico}

\usepackage{comment}

\usepackage[left=1.5in,top=1.0in,right=1.5in]{geometry}

\usepackage{amsthm}
\usepackage{amsmath}
\usepackage{amssymb}
\usepackage{tikz}
\usepackage{cite}
\usetikzlibrary{calc}
\usepackage{multirow}
\usepackage{lscape}

\usepackage{graphicx}

\newtheorem{theorem}{Theorem}[section]

\makeatletter
\newtheorem*{rep@theorem}{\rep@title}
\newcommand{\newreptheorem}[2]{%
\newenvironment{rep#1}[1]{%
 \def\rep@title{#2 \ref{##1}}%
 \begin{rep@theorem}}%
 {\end{rep@theorem}}}
\makeatother

\newreptheorem{theorem}{Theorem}

\newtheorem{lemma}[theorem]{Lemma}
\newtheorem{proposition}[theorem]{Proposition}

\theoremstyle{remark}

\newtheorem{remark}[theorem]{Remark}

\theoremstyle{definition}
\newtheorem{definition}[theorem]{Definition}

\newcommand{\R}{\mathcal{R}}

\newcommand{\x}{\mathbf{x}}

\newcommand{\1}{\mathbf{1}}

\newcommand{\A}{{\cal A}}

\newcommand{\CC}{\mathbb{C}}
\newcommand{\FF}{\mathbb{F}}

\newcommand{\NN}{{\mathbb{N}}}
\newcommand{\rk}{\textrm{rk}}
\newcommand{\RR}{\mathbb{R}}
\renewcommand{\SS}{\mathbb{S}}
\newcommand{\T}{\mathcal{T}}
\newcommand{\ZZ}{\mathbb{Z}}

\renewcommand{\span}{\textrm{span}}
\newcommand{\Hom}{\textrm{Hom}}
\renewcommand{\Im}{\textrm{Im }}

\newcommand{\blue}{\textcolor{blue}}

\begin{document}

\title{\textsf{The arithmetic Tutte polynomials \\
of the classical root systems.}}

\author{\textsf{Federico Ardila\footnote{\textsf{San Francisco State University, San Francisco, CA, USA and Universidad de Los Andes, Bogot\'a, Colombia. federico@sfsu.edu}}
 \qquad Federico Castillo\footnote{\textsf{Universidad de Los Andes, Bogot\'a, Colombia, and University of California, Davis, CA, USA. fcastillo@math.ucdavis.edu}}
 \qquad Michael Henley\footnote{\textsf{San Francisco State University, San Francisco, CA, USA. mhenley@sfsu.edu \newline This research was partially supported by the United States National Science Foundation CAREER Award DMS-0956178 (Ardila), the SFSU Math Department's National Science Foundation $\mathsf{(CM)^2}$ Grant DGE-0841164 (Henley), and  the SFSU-Colombia Combinatorics Initiative. This paper includes work from the second author's Los Andes undergraduate reseatch project and the third author's SFSU Master's thesis, carried out under the supervision of the first author.}}}}

\date{}

\maketitle

\begin{abstract}
Many combinatorial and topological invariants of a hyperplane arrangement can be computed in terms of its Tutte polynomial. Similarly, many invariants of a hypertoric arrangement can be computed in terms of its \emph{arithmetic} Tutte polynomial. 

We compute the arithmetic Tutte polynomials of the classical root systems $A_n, B_n, C_n,$ and $D_n$ with respect to their integer, root, and weight lattices. We do it in two ways: by introducing a \emph{finite field method} for arithmetic Tutte polynomials, and by enumerating signed graphs with respect to six parameters.
\end{abstract}

\noindent

\section{\textsf{Introduction}}

There are numerous constructions in mathematics which associate a combinatorial, algebraic, geometric, or topological object to a list of vectors $A$. It is often the case that important invariants of those objects (such as their size, dimension, Hilbert series, Betti numbers) can be computed directly from the \emph{matroid} of $A$, which only keeps track of the linear dependence relations between vectors in $A$. Sometimes such invariants depend only on the (arithmetic) Tutte polynomial of $A$, a two-variable polynomial defined below.

It is therefore of great interest to compute (arithmetic) Tutte polynomials of vector configurations. The first author \cite{Ardila} and Welsh and Whittle \cite{WW}
gave a \emph{finite-field method} for computing Tutte polynomials. In this paper we present an analogous method for computing arithmetic Tutte polynomials, which was also discovered by Br\"anden and Moci \cite{BM}. We cannot expect miracles from this method; computing Tutte polynomials is \#P-hard in general \cite{Welsh} and we cannot overcome that difficulty. However, this finite field method is extremely successful when applied to some vector configurations of interest.

Arguably the most important vector configurations in mathematics are the irreducible root systems, which play a fundamental role in many fields. The first author \cite{Ardila} used the finite field method to compute the Tutte polynomial of the classical root systems $\Phi = A_n, B_n, C_n, D_n$. De Concini and Procesi \cite{DPzonotope} and Geldon \cite{Geldon} computed it for the remaining root systems $E_6, E_7, E_8, F_4, G_2$. 

The main goal of this paper is to compute the \textbf{arithmetic} Tutte polynomial of the classical root systems $\Phi = A_n, B_n, C_n, D_n$. In doing so, we obtain combinatorial formulas for various quantities of interest, such as:
\\
\noindent $\bullet$ The volume and number of (interior) lattice points, of the zonotopes $Z(\Phi)$.
\\
\noindent $\bullet$ Various invariants associated to the hypertoric  arrangement $\T(\Phi)$ in a compact, complex, or finite torus.
\\
\noindent $\bullet$ 
The dimension of the Dahmen-Micchelli space $DM(\Phi)$ from numerical analysis.
\\
\noindent $\bullet$ The dimension of the De Concini-Procesi-Vergne space $DPV(\Phi)$ coming from index theory. 

Our results extend, recover, and in some cases simplify formulas of De Concini and Procesi \cite{DPzonotope}, Moci \cite{Mociroot}, and Stanley \cite{Stanleyzonotope} for some of these quantities.

Our formulas are given in terms of the \emph{deformed exponential function} of \cite{Sokal}, which is the following evaluation of the three variable Rogers-Ramanujan function:
\[
F(\alpha, \beta) = \sum_{n \geq 0} \frac{\alpha^n \, \beta^{n \choose 2}}{n!}.
\]
This function has been widely studied in complex analysis \cite{Langley, Liu, Morris} and statistical mechanics \cite{SS, SS2, Sokal}.

As a corollary, we obtain simple formulas for the characteristic polynomials of the classical root systems. In particular, we discover a surprising connection between the arithmetic characteristic polynomial of the root system $A_n$ and the enumeration of cyclic necklaces.

After introducing our finite-field method in Section \ref{sec:finitefield}, we obtain our formulas in two independent ways. In Section \ref{sec:compute1}, we apply our finite-field method to the classical root systems, reducing the computation to various enumerative problems over finite fields. In Section \ref{sec:compute2} we compute the desired Tutte polynomials by carrying out a detailed enumeration of (signed) graphs with respect to six different parameters. This enumeration may be of independent interest.
In Section \ref{sec:1var} we compute the arithmetic characteristic polynomials, and in Section \ref{sec:computations} we present a number of examples.

\subsection{\textsf{Preliminaries}}

\subsubsection{\textsf{Tutte polynomials and hyperplane arrangements}}

Given a vector configuration $A$ in a vector space $V$ over a field $\FF$, the \emph{Tutte polynomial} of $A$ is defined to be
\[
T_{A}(x,y)= \sum_{B \subseteq {A}} (x-1)^{r(A)-r(B)}(y-1)^{|B|-r(B)}
\]
where, for each $B \subseteq X$, the \emph{rank} of $B$ is $r(B) = \dim \span B$. The Tutte polynomial carries a tremendous amount of information about $A$. Three prototypical theorems are the following.

Let $V^*=\Hom(V,\FF)$ be the dual space of linear functionals from $V$ to $\FF$. Each vector $a \in A$ determines a normal hyperplane
\[
H_a = \{x \in V^* \, : \, x(a) = 0\}
\]
Let 
\[
\A(A) = \{H_a \, : \, a \in A\}, \qquad V(A) = V \, \setminus \bigcup_{H \in \A(A)} H
\]
be the \emph{hyperplane arrangement} of $A$ and its complement.
There is little harm in thinking of $\A(A)$ as the arrangement of hyperplanes perpendicular to the vectors of $A$, but the more precise definition will be useful in the next section.

\begin{theorem}\blue{$(\FF = \RR)$} (Zaslavsky)
\cite{Zaslavsky}
Let $\A(A)$ be a \textbf{real} hyperplane arrangement in $\RR^n$. The complement $V(A)$ consists of $|T_A(2,0)|$ regions.
\end{theorem}

\begin{theorem}\blue{$(\FF = \CC)$} (Goresky-MacPherson, Orlik-Solomon)
\cite{GM, OS}
Let $\A(A)$ be a \textbf{complex} hyperplane arrangement in $\CC^n$. The cohomology ring of the complement $V(A)$ has Poincar\'e polynomial
\[
\sum_{k \geq 0} \mathrm{rank } \, H^k(V(A), \ZZ) q^k = (-1)^r q^{n-r} T_A(1-q,0).
\]
\end{theorem}

\begin{theorem}\label{thm:finfieldhyps}\blue{($\FF = \FF_q$: Finite field method)} (Crapo-Rota, Athanasiadis, Ardila, Welsh-Whittle)
\cite{Ardila, Athanasiadis, CrapoRota, WW}
Let $\A(A)$ be a hyperplane arrangement in $\FF_q^n$ where $\FF_q$ is the finite field of $q$ elements for a prime power $q$. Then the complement $V(A)$ has size
\[
| V(A) | = (-1)^r q^{n-r} T_A(1-q,0)
\]
and, furthermore,
\[
\sum_{p \in \FF_q^n} t^{h(p)} = (t-1)^r q^{n-r} T_A \left( \frac{q+t-1}{t-1}, t \right)
\]
where $h(p)$ is the number of hyperplanes of $\A(A)$ that $p$ lies on. 
\end{theorem}
The first statement was proved by Crapo and Rota \cite{CrapoRota}, and Athanasiadis \cite{Athanasiadis} used it to compute the characteristic polynomial 
$(-1)^r q^{n-r}T_A(1-q,0)$ of various arrangements $A$ of interest. These results were extended to Tutte polynomials by the first author \cite{Ardila} and by Welsh and Whittle \cite{WW}.

\subsubsection{\textsf{Arithmetic Tutte polynomials and hypertoric arrangements}}

If our vector configuration $A$ lives \textbf{in a lattice} $\Lambda$, then the \emph{arithmetic Tutte polynomial} is
\[
M_{A}(x,y)= \sum_{B \subseteq {A}} m(B)(x-1)^{r(A)-r(B)}(y-1)^{|B|-r(B)}
\]
where, for each $B \subseteq A$, the \emph{multiplicity} $m(B)$ of $B$ is the index of $\ZZ B$ as a sublattice of $\span B \cap \Lambda$. The arithmetic Tutte polynomial also carries a great amount of information about $A$, but it does so in the context of \emph{toric arrangements}.

\begin{definition} Let $T=\Hom(\Lambda,\FF^*)$ be the character group, consisting of the group homomorphisms from $\Lambda$ to the multiplicative group $\FF^*=\FF\backslash \{0\}$ of the field $\FF$. We might also consider the unitary characters $T=\Hom(\Lambda,\SS^1)$ where $\SS^1$ is the unit circle in $\CC$. It is easy to check that $T$ is isomorphic to $\FF^*$ and to $\SS^1$, respectively. 

Each element $a \in A$ determines a hypertorus
\[
T_a = \{t \in T \, : \, t(a) = 1\}
\]
in $T$. For instance $a=(2,-3,5)$ gives the hypertorus $x^2y^{-3}z^5=1$. Let 
\[
\T(A) = \{T_a \, : \, a \in A\}, \qquad R(A) = T \, \setminus \bigcup_{T \in \T(A)} T
\]
be the \emph{toric arrangement} of $A$ and its complement, respectively.
\end{definition}


\begin{theorem}\cite{ERS, MociTutte} 
\blue{$(\FF^* = \SS^1)$} (Moci) Let $\T(A)$ be a \textbf{real} toric arrangement in the compact torus $T \cong (\SS^1)^n$. Then $\R(A)$ consists of $|M_A(1,0)|$ regions.
\end{theorem}

\begin{theorem}\cite{DPtoric, MociTutte}
\blue{$(\FF^* = \CC^*)$} (Moci) Let $\T(A)$ be a \textbf{complex} toric arrangement in the torus $T \cong (\CC^*)^n$. The cohomology ring of the complement $\R(A)$ has Poincar\'e polynomial
\[
\sum_{k \geq 0} \mathrm{rank } \, H^k(\R(A), \ZZ) q^k = q^n M_A\left(\frac{2q+1}q,0\right)
\]
\end{theorem}

For finite fields we prove the following result, which is also one of our main tools for computing arithmetic Tutte polynomials.

\begin{theorem}\label{thm:finfield}
\blue{($\FF^* = \FF_{q+1}^*$: Finite field method)}
Let $\T(A)$ be a toric arrangement in the torus $T \cong (\FF_{q+1}^*)^n$ where $\FF_{q+1}$ is the finite field of $q+1$ elements for a prime power $q+1$. Assume that $m(B) | q $ for all $B \subseteq A$.
Then the complement $\R(A)$ has size
\[
|\R(A) | = (-1)^r q^{n-r} M_A(1-q,0)
\]
and, furthermore,
\[
\sum_{p \in T} t^{h(p)} = (t-1)^r q^{n-r} M_A \left( \frac{q+t-1}{t-1}, t \right)
\]
where $h(p)$ is the number of hypertori of $\T(A)$ that $p$ lies on. \end{theorem}

The first part of Theorem \ref{thm:finfield} is equivalent to a recent result of  Ehrenborg, Readdy, and Sloane\cite[Theorem 3.6]{ERS}. A multivariate generalization of the second part was obtained simultaneously and independently by Br\"anden and Moci \cite{BM} -- they consider target groups other than $\FF_{q+1}^*$, but for our purposes this choice will be sufficient.
 
The second statement of Theorem \ref{thm:finfield} is significantly stronger than the first because it involves two different parameters; so if we are able to compute the left hand side, we will have computed the whole arithmetic Tutte polynomial. For that reason, we regard this as a \emph{finite field method} for arithmetic Tutte polynomials.

There are several other reasons to care about the arithmetic Tutte polynomial of $A$; we refer the reader to the references for the relevant definitions.

\begin{theorem}
Let $A$ be a vector configuration in a lattice $\Lambda$.
\begin{itemize}
\item The volume of the zonotope $Z(A)$ is $M_A(1,1)$. \cite{Stanleyzonotope}.
\item The Ehrhart polynomial of the zonotope $Z(A)$ is $q^nM(1+\frac1q,1)$. \cite{Stanleyzonotope, MociEhrhart}
\item The dimension of the Dahmen-Micchelli space $DM(A)$ is $M_A(1,1)$. \cite{DPV1, MociTutte}
\item The dimension of the De Concini-Procesi-Vergne space $DPV(A)$ is $M_A(2,1)$. \cite{DPV1, Mocigeom}
\end{itemize}
\end{theorem}

\subsubsection{\textsf{Root systems and lattices}}

Root systems are arguably the most fundamental vector configurations in mathematics. Accordingly, they play an important role in the theory of hyperplane arrangements and toric arrangements. In fact, the construction of the arithmetic Tutte polynomial was largely motivated by this special case. \cite{DPzonotope, Mociroot}
We will pay special attention to the four infinite families of finite root systems, known as the \emph{classical root systems}:
\begin{eqnarray*}
A_{n-1} &=& \{e_i-e_j,\, : \, 1\leq i < j\leq n\} \\
B_n &=& \{e_i -  e_j, e_i + e_j \, : \,  1\leq i <  j\leq n\} \cup \{e_i \, : \, 1 \leq i \leq n\} \\
C_n &=& \{e_i -  e_j, e_i + e_j \, : \,  1\leq i <  j\leq n\} \cup \{2e_i \, : \, 1 \leq i \leq n\} \\
D_n &=& \{e_i -  e_j, e_i + e_j \, : \,  1\leq i <  j\leq n\} 
\end{eqnarray*}
Notice that we are only considering the positive roots of each root system. It is straightfoward to adapt our methods to compute the (arithmetic) Tutte polynomials of the full root systems.

We refer the reader to \cite{Bj05} or \cite{Hu90} for an introduction to root systems
and Weyl groups, and  \cite[Chapter 6]{Or92}, \cite{Sthyps} for more information on Coxeter arrangements.

The arithmetic Tutte polynomial of a vector configuration $A$ depends on the lattice where $A$ lives. For a root system $\Phi$ in $\RR^v$ there are at least three natural choices: the integer lattice $\ZZ^v$, the weight lattice $\Lambda_W$, and the root lattice $\Lambda_R$. The second is the lattice generated by the roots, while the third is the lattice generated by the fundamental weights. The root lattices and weight lattices of the classical root systems are the following \cite{FH}:
\begin{eqnarray*}
\Lambda_W(A_{n-1}) &=& \ZZ\{e_1, \ldots, e_n\} / (\varSigma \, e_i=0)  \\
\Lambda_R(A_{n-1}) &=& \{\varSigma \, a_ie_i \, : \, a_i \in \ZZ, \varSigma a_i=0\} / (\varSigma e_i=0) \\
\Lambda_W(B_n) &=& \ZZ\{e_1, \ldots, e_n, (e_1+\cdots+e_n)/2\}\\
\Lambda_R(B_n) &=& \ZZ\{e_1, \ldots, e_n\}\\
\Lambda_W(C_n) &=& \ZZ\{e_1, \ldots, e_n\}\\
\Lambda_R(C_n) &=&  \{\varSigma \, a_ie_i \, : \, a_i \in \ZZ, \varSigma a_i \textrm{ is even}\}\\
\Lambda_W(D_n) &=& \ZZ\{e_1, \ldots, e_n, (e_1+\cdots+e_n)/2\}\\
\Lambda_R(D_n) &=&  \{\varSigma \, a_ie_i \, : \, a_i \in \ZZ, \varSigma a_i \textrm{ is even}\}
\end{eqnarray*}

For example, we are considering five different 2-dimensional lattices. The weight lattice of $A_2$ is the triangular lattice, while its root lattice is an index 3 sublattice inside it. The usual square lattice $\ZZ\{e_1, e_2\}$ contains the root lattice of $C_2$ and $D_2$ as an index 2 sublattice, and it is contained in the weight lattice of $B_2$ and $D_2$ as an index 2 sublattice.


More generally, we have the following relation between the different lattices:
\begin{proposition}\cite[Lemma 23.15]{FH}
The root lattice $\Lambda_R$ is a sublattice of the weight lattice $\Lambda_W$, and the index $[\Lambda_W \, : \, \Lambda_R]$ equals the determinant $\det A_\Phi$ of the Cartan matrix. For the classical root systems, this is:
\[
\det A_{n-1} = n, \qquad 
\det B_n = 2, \qquad
\det C_n = 2, \qquad
\det D_n = 4 
\]
where the last formula holds only for $n \geq 3$.
\end{proposition}

\subsection{\textsf{Our formulas}}

We give explicit formulas for the arithmetic Tutte polynomials of the classical root systems. Our results are most cleanly expressed in terms of the \emph{(arithmetic) coboundary polynomial}, which is the following simple transformation of the (arithmetic) Tutte polynomial:
\[
\overline{\chi}_{\A}(X,Y) = (y-1)^{r(\A)} T_{\A}(x,y), \qquad {\psi}_{\A}(X,Y) = (y-1)^{r(\A)} M_{\A}(x,y)
\]
where 
\[
x = \frac{X+Y-1}{Y-1}, \quad y=Y, \qquad \textrm{ and } \qquad X=(x-1)(y-1), \quad Y=y.
\]
Clearly, the (arithmetic) Tutte polynomial can be recovered readily from the (arithmetic) coboundary polynomial. Throughout the paper, we will continue to use the variables $X,Y$ for coboundary polynomials and $x,y$ for Tutte polynomials.

Our formulas are conveniently expressed in terms of the exponential generating functions for the coboundary polynomials:

\begin{definition}\label{def:Tuttegen}
For the infinite families $\Phi=B,C,D$,
of classical root systems, let the \emph{Tutte generating function} and the \emph{arithmetic Tutte generating function}\footnote{It might be more accurate to call it the \emph{arithmetic coboundary generating function}, but we prefer this name because the Tutte polynomial is much more commonly used than the coboundary polynomial.} be
\[
\overline{X}_\Phi(X,Y,Z) = \sum_{n \geq 0} \overline{\chi}_{\Phi_n}(X,Y) \frac{Z^n}{n!}, \qquad
\Psi_\Phi(X,Y,Z) = \sum_{n \geq 0} \psi_{\Phi_n}(X,Y) \frac{Z^n}{n!},
\]
respectively; and for $\Phi=A$ let them be
\[
\overline{X}_A(X,Y,Z) = 1+X \sum_{n \geq 1} \overline{\chi}_{A_{n-1}}(X,Y) \frac{Z^n}{n!}, \quad
\Psi_A(X,Y,Z) = 1+X\sum_{n \geq 1} \psi_{A_{n-1}}(X,Y) \frac{Z^n}{n!}.
\]
For $\Phi=A$ we need the extra factor of $X$, since the root system $A_{n-1}$ is of rank $n-1$ inside $\ZZ^n$.
\end{definition}

Our formulas are given in terms of the following functions which have been studied extensively in complex analysis \cite{Langley, Liu, Morris} and statistical mechanics \cite{SS, SS2, Sokal}:

\begin{definition} \label{def:RR}
Let the \emph{three variable Rogers-Ramanujan function} be
\[
\widetilde{R}(\alpha, \beta, q) = \sum_{n \geq 0} \frac{\alpha^n \,  \beta^{n \choose 2}}{(1+q)(1+q+q^2) \cdots (1+q+\cdots+q^{n-1})}
\]
and the \emph{deformed exponential function} be
\[
F(\alpha, \beta) = \sum_{n \geq 0} \frac{\alpha^n \, \beta^{n \choose 2}}{n!} = \widetilde{R}(\alpha, \beta, 1).
\] 
\end{definition}

%
%
%
%
%

We denote the  arithmetic Tutte generating functions of the root systems with respect to the integer, weight, and root lattices by 
$\Psi_\Phi, \Psi_\Phi^W,$ and $\Psi_\Phi^R$, respectively. Tutte (in Type $A$) and the first author (in types $A, B, C, D$) computed the ordinary Tutte generating functions for the classical root systems:


\begin{theorem} \label{thm:Tutte} \cite{Ardila, Tutte} The Tutte generating functions of the classical root systems are
\begin{eqnarray*}
X_A &=& F(Z,Y)^X\\
X_B &=& F(2Z,Y)^{(X-1)/2}F(YZ,Y^2)\\
X_C &=& F(2Z,Y)^{(X-1)/2}F(YZ,Y^2)\\
X_D &=& F(2Z,Y)^{(X-1)/2}F(Z,Y^2)
\end{eqnarray*}
\end{theorem}
De Concini and Procesi \cite{DPzonotope} and Geldon \cite{Geldon} extended those computations to the exceptional root systems $G_2, F_4, E_6, E_7,$ and $E_8$. 

In this paper we compute the \textbf{arithmetic} Tutte polynomials of the classical root systems. Our main results are the following:

\begin{theorem}  \label{thm:TutteZ} The arithmetic Tutte generating functions of the classical root systems \textbf{in their integer lattices} are
\begin{eqnarray*}
\Psi_A &=& F(Z,Y)^X\\
\Psi_B &=& F(2Z,Y)^{\frac{X}2-1}F(Z,Y^2)F(YZ,Y^2)\\
\Psi_C &=& F(2Z,Y)^{\frac{X}2-1}F(YZ,Y^2)^2\\
\Psi_D &=& F(2Z,Y)^{\frac{X}2-1}F(Z,Y^2)^2
\end{eqnarray*}
\end{theorem}

\begin{theorem}  \label{thm:TutteR} The arithmetic Tutte generating functions of the classical root systems \textbf{in their root lattices} are
\begin{eqnarray*}
\Psi^R_A &=& F(Z,Y)^X\\
\Psi^R_B &=& F(2Z,Y)^{\frac{X}2-1}F(Z,Y^2)F(YZ,Y^2)\\
\Psi^R_C &=& \frac12 F(2Z,Y)^{\frac{X}2-1} \left[F(2Z,Y)+F(YZ,Y^2)^2\right]\\
\Psi^R_D &=& \frac12 F(2Z,Y)^{\frac{X}2-1}\left[F(2Z,Y) + F(Z,Y^2)^2\right]\\
\end{eqnarray*}
\end{theorem}

\begin{theorem} \label{thm:TutteW} The arithmetic Tutte generating functions of the classical root systems \textbf{in their weight lattices} are
\[
\Psi^W_A = 
\sum_{n \in \NN} \varphi(n)\left(\left[F(Z,Y)F(\omega_n Z, Y) F(\omega_n^2 Z, Y) \cdots F(\omega_n^{n-1} Z, Y)\right]^{X/n}-1 \right)
\]
where $\varphi(n) = \#\{m \in \NN \, : \, 1 \leq m \leq n, \,\, (m,n)=1\}$ is Euler's totient function and $\omega_n$ is a primitive $n$th root of unity for each $n$, 
\begin{eqnarray*}
\Psi^W_B &=& F(2Z,Y)^{\frac{X}4-1}F(Z,Y^2)F(YZ,Y^2)\left[F(2Z,Y)^{\frac{X}4} + F(-2Z,Y)^{\frac{X}4}\right]
\\
\Psi^W_C &=& F(2Z,Y)^{\frac{X}2-1}F(YZ,Y^2)^2\\
\Psi^W_D &=& F(2Z,Y)^{\frac{X}4-1}F(Z,Y^2)^2\left[F(2Z,Y)^{\frac{X}4} + F(-2Z,Y)^{\frac{X}4}\right]
\end{eqnarray*}
\end{theorem}

%
%

\begin{remark}
The arithmetic Tutte polynomials of type $A$ in the weight lattice are more subtle than the other ones, due to the large index $[\Lambda_W : \Lambda_R]$ in that case. While the formula of Theorem \ref{thm:TutteW} seems rather impractical for computations, at the end of Section \ref{sec:compute2} we give an  alternative formulation which is efficient and easily implemented.
\end{remark}

\begin{remark}
The generating function for the actual Tutte polynomials is obtained easily from the above by substituting
\[
X = (x-1)(y-1), \qquad Y=y, \qquad Z=\frac{z}{y-1}.
\]
For instance, the formula for $\Psi_C^W$ above can be rewritten as:
\[
\sum_{v\geq 0}M^W_{C_v}(x,y) \frac{z^v}{v!}
=
F\left(\frac{2z}{y-1}, y\right)^{\frac12({(x-1)(y-1)-1})}
F\left(\frac{yz}{y-1}, y^2\right)^2.
\]
\end{remark}

This also allows us to give formulas for the respective arithmetic characteristic polynomials
\[
\chi^\Lambda_A(q)= (-1)^r q^{n-r} M^\Lambda_A(1-q,0).
\]
Some representative formulas are the following:

\begin{theorem}\label{thm:charZ} The arithmetic characteristic polynomials of the classical root systems \textbf{in their integer lattices} are
\begin{eqnarray*}
\chi^\ZZ_{A_{n-1}}(q) &=& q(q-1)\cdots (q-n+1)\\
\chi^\ZZ_{B_n}(q) &=& (q-2)(q-4)(q-6)\cdots (q-2n+4)(q-2n+2)(q-2n) \\
\chi^\ZZ_{C_n}(q) &=& (q-2)(q-4)(q-6)\cdots (q-2n+4)(q-2n+2) (q-n) \\
\chi^\ZZ_{D_n}(q) &=& (q-2)(q-4)(q-6)\cdots (q-2n+4) (q^2 - 2(n-1)q + n(n-1))
\end{eqnarray*}
\end{theorem}

These are similar but not equal to the classical characteristic polynomials of the root systems. \cite{Athanasiadis, Sthyps} The following formula is quite different from the classical one.

\begin{theorem}\label{thm:charWA} The arithmetic characteristic polynomials of the root systems $A_{n-1}$ \textbf{in their weight lattices} are given by
\[
\chi^W_{A_{n-1}}(q) =  \frac{n!}q \sum_{m | n} (-1)^{n-\frac nm} \varphi(m) {q/m \choose n/m}
\]
In particular, when $n\geq 3$ is prime, 
\[
\chi^W_{A_{n-1}}(q) = (q-1)(q-2)\cdots (q-n+1) + (n-1)(n-1)!.
\]
\end{theorem}

When  $n$ is odd and $n|q$ we obtain an intriguing combinatorial interpretation:

\begin{theorem}\label{cor:charW}
If $n,q$ are integers with $n$ odd and $n|q$, then $\chi^W_{A_{n-1}}(q)/n!$ equals the number of cyclic necklaces with $n$ black beads and $q-n$ white beads.
\end{theorem}

\subsection{\textsf{Comparing the two methods.}}

For all but one of the formulas above, we will give one ``finite field" proof and one ``graph enumeration" proof. Each method has its advantages.
When the underlying lattice is $\ZZ^n$, the finite field method seems preferrable, as it gives more straightforward proofs than the graph enumeration method. However, this is no longer the case with more complicated lattices. In particular, we only have one proof for the formula for $\Psi_A^W$, using graph enumeration. There should also be a ``finite field method" proof for this result, but it seems more difficult and less natural.

\subsection{{\textsf{An example: $\mathbf{\mathsf{C_2}}$.}}}\label{sec:C_2}
Before going into the proofs, we carry out an example. 
Consider the root system $C_2 = \{2e_1, e_1+e_2, 2e_2, e_1-e_2\}$ in $\ZZ^2$. This vector configuration is drawn in red in Figure \ref{fig:C2} with its associated zonotope
\[
Z(C_2) = \{a(2e_1) + b(e_1+e_2) + c(2e_2) + d(e_1-e_2) \, : \, 0 \leq a,b,c,d \leq 1\}.
\]
Figure \ref{fig:C2} also shows the two natural lattices for $C_2$: $\ZZ^2 = \Lambda_W(C_2)$ and its index 2 sublattice $\Lambda_R(C_2)$. 
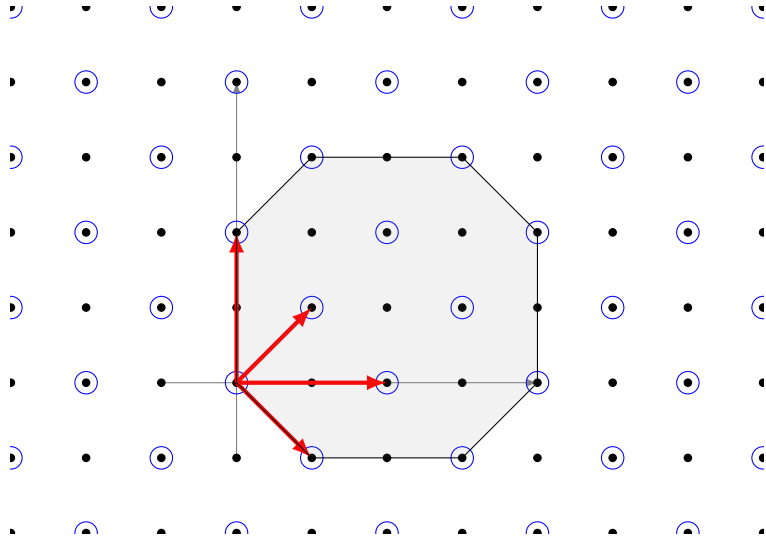
\begin{figure}[ht]\label{fig:C2}
  \centering
  \begin{tikzpicture}
    \coordinate (Origin)   at (0,0);
    \coordinate (XAxisMin) at (-1,0);
    \coordinate (XAxisMax) at (4,0);
    \coordinate (YAxisMin) at (0,-1);
    \coordinate (YAxisMax) at (0,4);
    \draw [thin, gray,-latex] (XAxisMin) -- (XAxisMax);
    \draw [thin, gray,-latex] (YAxisMin) -- (YAxisMax);

    \clip (-3,-2) rectangle (7cm,5cm); 
    \pgftransformcm{1}{0}{0}{1}{\pgfpoint{0cm}{0cm}}

    \foreach \x in {-7,-6,...,7}{
      \foreach \y in {-7,-6,...,7}{
        \node[draw,circle,inner sep=3 pt, blue] at (2*\x,2*\y) {};
            
      }
    }
    \foreach \x in {-7,-6,...,7}{
      \foreach \y in {-7,-6,...,7}{
        \node[draw,circle,inner sep=3pt, blue] at (2*\x+1,2*\y+1) {};
            
      }
    }
    
    \foreach \x in {-7,-6,...,7}{
      \foreach \y in {-7,-6,...,7}{
        \node[draw,circle,inner sep=1pt, fill] at (\x,\y) {};

      }
    }

    \draw [ultra thick,-latex,red] (Origin)
        -- ($(0,2)$) node [above left] {};
   \draw [ultra thick,-latex,red] (Origin)
       -- ($(2,0)$) node [below right] {};
   \draw [ultra thick,-latex,red] (Origin)
      -- ($(1,1)$) node [below right] {};
   \draw [ultra thick,-latex,red] (Origin)
      -- ($(1,-1)$) node [above right] {};
    \draw [thin,-latex,black, fill=gray, fill opacity=0.1] (0,0)
         -- ($(0,2)$)
         -- ($(1,3)$) -- ($(3,3)$) 
         -- ($(4,2)$) -- ($(4,0)$)
         -- ($(3,-1)$) -- ($(1,-1)$) -- cycle;
  \end{tikzpicture}
  \caption{The vector configuration $C_2=\{(0,2),(1,1),(2,0),(1,-1)\}$ and the associated zonotope $Z(C_2)$.}
  \label{Example}
\end{figure}

Let us compute the arithmetic Tutte polynomial $M_{C_2}(x,y)$ with respect to $\ZZ^2$:
\noindent $\bullet$ The empty subset has multiplicity 1 and hence contributes a $(x-1)^{2}$ term. \\
\noindent $\bullet$
The singletons $\{(2,0)\}$ and $\{(0,2)\}$ have multiplicity 2 and $\{1,1\}$ and $\{1,-1\}$ have multiplicity 1, contributing $2(x-1)+2(x-1)+(x-1)+(x-1)=6(x-1)$. \\
\noindent $\bullet$ All pairs are bases. The pair $\{(2,0),(0,2)\}$ has multiplicity $4$, and the other 5 pairs have multiplicity 2, so we get a total contribution of $5\cdot2+4=14.$\\
\noindent $\bullet$ Each triple has rank 2 and multiplicity 2, for a contribution of $4 \cdot 2(y-1)=8(y-1).$\\
\noindent $\bullet$
The whole set contributes  $2(y-1)^2$.

Therefore the arithmetic Tutte polynomial of $C_2$ in $\ZZ^2$ is
\begin{eqnarray*}
M_{C_2}(x,y)&=&(x-1)^2+6(x-1)+14+8(y-1)+2(y-1)^2\\
&=&x ^2+ 2y^2 + 4x + 4y + 3.
\end{eqnarray*}

\noindent This predicts that the Ehrhart polynomial of the zonotope of $C_2$ is
\[
E_{Z(C_2)}(t)=t^2M_{C_2}(1+1/t,1)=14t^2+6t^2+1
\]
which in turn predicts that the zonotope has area $14$, $14+6+1=21$ lattice points, and $14-6+1=9$ interior lattice points.

Consider instead the arithmetic Tutte polynomial with respect to the root lattice:
\begin{eqnarray*}
M^R_{C_2}(x,y)&=&(x-1)^2+4(x-1)+7+4(y-1)+(y-1)^2\\
&=&x ^2+ y^2 + 2x + 2y + 1.
\end{eqnarray*}
\noindent Now the Ehrhart polynomial is
\[
E_{Z(C_2)}^R(t)=t^2M^R_{C_2}(1+1/t,1)=7t^2+4t+1
\]
which in turn predicts that the zonotope has area $7$, $7+4+1=12$ lattice points, and $7-4+1=4$ interior lattice points with respect to the root lattice.

%

\section{\textsf{The finite field method for hypertoric arrangements}}\label{sec:finitefield}

Recall that the \emph{arithmetic Tutte polynomial} of a vector configuration $A \subseteq \Lambda$ is
\[
M_{A}(x,y)= \sum_{B \subseteq {A}} m(B)(x-1)^{r(A)-r(B)}(y-1)^{|B|-r(B)}
\]
where, for each $B \subseteq A$, $m(B)$ is the index of $\ZZ B$ as a sublattice of $\span B \cap \Lambda$.

\subsection{\textsf{The finite field method for hypertoric arrangements: the proof.}}\label{sec:finitefieldproof}

We start by restating Theorem \ref{thm:finfield} more explicitly. We now omit the first statement in the previous formulation, which follows from the second one by setting $t=0$.

\begin{reptheorem}{thm:finfield}
\blue{($\FF^* = \FF_{q+1}^*$: Finite field method)}
Let $A$ be a collection of vectors in a lattice $\Lambda$ of rank $n$. Let $q+1$ be a prime power such that $m(B) | q $ for all $B \subseteq A$ and consider the torus $T = \Hom(\Lambda,\FF_{q+1}^*) \cong (\FF_{q+1}^*)^n$. Let $\T(A)$ be the corresponding arrangement of hypertori in $T$.
Then 
\[
\sum_{p \in T} t^{h(p)} = (t-1)^r q^{n-r} M_A \left( \frac{q+t-1}{t-1}, t \right) 
\]
where $h(p)$ is the number of hypertori of $\T(A)$ that $p$ lies on. \end{reptheorem}

\textit{Remark}: 
The finite field method is cleanly expressed in terms of the arithmetic coboundary polynomial:
\[
\sum_{p \in (\FF_{X+1}^*)^n} Y^{h(p)} = X^{n-r} \psi(X,Y)
\]
whenever $X+1$ is a prime power with $m(B)|X$ for all $B \subseteq A$.

\begin{proof}[Proof of Theorem \ref{thm:finfield}]
The key observation is the following:

\begin{lemma}\label{finlem}
For any $B\subseteq \Lambda$ and any $q$ such that $m(B) | q$, we have
$
\left| \bigcap_{b\in B} T_b \right| = m(B)q^{n-\rk(B)}, 
$
where 
$
T_b = \{t \in T \, : \, t(b) = 1\}
$
is the hypertorus associated to $b$.
\end{lemma}

\begin{proof}[Proof of Lemma \ref{finlem}]
An element $t\in  \bigcap_{b\in B} T_b$ is a homomorphism $t:\Lambda \rightarrow \FF_{q+1}^*$ such that $B \subseteq \textrm{ker}(t)$; this is equivalent to $\ZZ B \subseteq \textrm{ker}(t)$.
Those maps are in bijection with the maps $\hat{t}:\Lambda/\ZZ B\rightarrow \FF_{q+1}^*$, and we proceed to enumerate them.

By the Fundamental Theorem of Finitely Generated Abelian Groups, 
%
%
we can write $\Lambda / \ZZ B$ uniquely as 
\[
\Lambda/\ZZ B\cong  \ZZ^{n-\rk B} \times \ZZ/d_1\times \ZZ/d_2\times \cdots \ZZ/d_k
\]
where $d_1|d_2|\cdots|d_k$ are its invariant factors. Notice that $m(B) = d_1d_2\cdots d_k$.

Our desired map is determined by maps
\[
t_0: \ZZ^{n - \rk B} \rightarrow \FF_{q+1}^* \quad \textrm{ and } \quad
t_i: \ZZ/d_i \rightarrow \FF_{q+1}^* \quad \textrm{ for } {1 \leq i \leq k}.
\]
for the individual factors. There are $q^{n-\rk B}$ choices for the map $t_0$. 
Since $\FF_{q+1}^* \cong \ZZ/q$ and $d_i$ divides $m(B)$ which in turn divides $q$, there are exactly $d_i$ homomorphisms $\ZZ/d_i \rightarrow \FF_{q+1}^*$ for each $i$. Therefore the number of homomorphisms we are looking for is $q^{n-\rk B} d_1d_2 \cdots d_k  = q^{n-\rk B} m(B)$.
\end{proof}

We are now ready to complete the proof of Theorem \ref{thm:finfield}. 
For each $p \in T$, let $H(p)$ be the set of hypertori of $\T(A)$ in which $p$ is contained, so $h(p)= |H(p)|$.  
Then we have
\begin{eqnarray*}
(t-1)^r q^{n-r}M_{A}\left(\frac{q+t-1}{t-1},t\right)
&=& \displaystyle \sum_{B \subseteq {A}}m(B)q^{n-r(B)}(t-1)^{|B|}  \\
&=&  \displaystyle \sum_{B \subseteq {A}}\left|\bigcap_{b\in B} T_b \right|(t-1)^{|B|}  \\
&=&  \displaystyle \sum_{B \subseteq {A}}\sum_{\stackrel{p \in T_b}{\textrm{ for all } b \in B}}(t-1)^{|B|}  \\    
&=&  \displaystyle \sum_{p\in (\mathbb{F}_{q+1}^*)^m}\sum_{B \subseteq  H(p)}(t-1)^{|B|}  \\
&=&  \displaystyle \sum_{p\in (\mathbb{F}_{q+1}^*)^m}(1+(t-1))^{h(p)}
\end{eqnarray*}
as desired.
\end{proof}

\textit{Remark}: Dirichlet's theorem on primes in arithmetic progressions \cite{Dirichlet} guarantees that for any $A \subseteq \Lambda$ there are infinitely many primes $q+1$ which satisfy the hypothesis of Theorem \ref{thm:finfield}. In particular, by computing the left hand side of Theorem \ref{thm:finfield} for enough such primes, we can obtain $M_A(x,y)$ by polynomial interpolation.

%

\textit{Remark}: The choice of $q$ is critical, and different choices of $q$ will give different answers. For example, another case of interest is when $q$ is a prime power such that $m(B)|q-1$ for all $B \subseteq A$. As Br\"anden and Moci point out \cite{BM}, 
if each $d_i$ divides $q-1$, then each $d_i$ is relatively prime with $q$; so in the proof of Lemma \ref{finlem}, there will be just one trivial choice for each $t_i$. Therefore, in this case the left hand side of Theorem \ref{thm:finfield} computes the \emph{classical} coboundary and Tutte polynomials.
\\

\section{\textsf{Computing Tutte polynomials using the finite field method}}
\label{sec:compute1}


In this section we apply Theorem \ref{thm:finfield}, the finite field method for hypertoric arrangements, to give formulas for all but one of the arithmetic Tutte polynomials of the classical root systems with respect to the integer, root, and weight lattices, proving Theorems \ref{thm:TutteZ}, \ref{thm:TutteR}, \ref{thm:TutteW}. The procedure will be similar to the one used in \cite{Ardila} for classical Tutte polynomials, but the arithmetic features of this computation will require new ideas. We proceed in increasing order of difficulty.

\begin{reptheorem}{thm:TutteZ}
The arithmetic Tutte generating functions of the classical root systems \textbf{in their integer lattices} are
\begin{eqnarray*}
\Psi_A &=& F(Z,Y)^X\\
\Psi_B &=& F(2Z,Y)^{\frac{X}2-1}F(Z,Y^2)F(YZ,Y^2)\\
\Psi_C &=& F(2Z,Y)^{\frac{X}2-1}F(YZ,Y^2)^2\\
\Psi_D &=& F(2Z,Y)^{\frac{X}2-1}F(Z,Y^2)^2
\end{eqnarray*}
\end{reptheorem}

\begin{proof}
Here $\Lambda = \ZZ^n$ and $T_{\ZZ}=\textrm{Hom}(\ZZ^n, \FF_{q+1}^*)$ is isomorphic to $(\FF_{q+1}^*)^n$. Each of the hypertori is a solution to a multiplicative equation.  For example, $T_{e_1+e_2}$ is the set of solutions to $x_1x_2=1$. We need to enumerate points in 
$(\FF_{q+1}^*)^n$ by the number of hypertori that contain them.

\medskip
\noindent 
\textbf{\textsf{Type $\mathsf{A}$}}: 
First we prove the formula when $X+1=q+1$ is prime. We need to compute $\psi_{A_n}(X,Y)$, and we use the finite field method.
For each $p \in (\FF_{X+1}^*)^n$ let $P_k=\{j \in [n] \, : \, a_j = k\}$ for $k=1, 2, \ldots, X$. This is a bijection between points in $(\FF_{X+1}^*)^n$ and ordered partitions $[n]=P_1 \sqcup \cdots \sqcup P_X$. Furthermore, $h(p)$ is the number of pairs of equal coordinates of $p$, so we have $h(p)= 
{|P_1| \choose 2} + \cdots + {|P_X| \choose 2}$. The finite field method then says that for $n \geq 1$
\[
X\psi_{A_{n-1}}(X,Y)= \sum_{[n] = P_1 \sqcup \cdots \sqcup P_X} Y^{{|P_1| \choose 2} + \cdots + \binom{|P_X|}{2}}.
\]
Notice that this equation holds for \textbf{any} prime $X+1$, since the collection $A_{n-1}$ is unimodular in $\ZZ^n$, so $m(B)=1$ for all subsets of it. The compositional formula for exponential generating functions \cite[Theorem 5.1.4]{EC2} then gives
\begin{equation}\label{eq}
1+X\sum_{n \geq 1} \psi_{A_{n-1}}(X,Y) \frac{Z^n}{n!} = 
\left(\sum_{n \geq 0} Y^{n \choose 2} \frac{Z^n}{n!}\right)^X.
\end{equation}
Having established this equation whenever $X+1$ is prime, we now need to prove it as an equality of formal power series in $X,Y,Z$. To do it, we observe that in each monomial on either side, the $X$-degree is less than or equal to the $Z$-degree. On the left-hand, side this follows from the fact that the $X^{n-r}\psi(X,Y) = \sum_B m(B)X^{n-r(B)}(Y-1)^{|B|}$ has $X$-degree equal to $n$. On the right-hand side, which has the form $(1+ZG(Y,Z))^X$, this follows from the binomial theorem.

We conclude that, for any fixed $a$ and $b$, the coefficients of $Y^aZ^b$ on both sides of (\ref{eq}) are polynomials in $X$. Since they are equal for infinitely many values of $X$, they are equal as polynomials. The desired result follows.

\medskip
\noindent 
\textbf{\textsf{Type $\mathsf{B}$}}: 
Again let $X+1=q+1$ be prime and now assume that $X$ is a multiple of $m(B)$ for all $B \subseteq B_n$ for a particular $n$.
%
%
%
%
%
 Split $\FF_{X+1}^*$ into singletons or pairs containing an element and its inverse. In other words, choose $a_1=1,a_2=-1, a_3, a_4,\ldots,a_{X/2+1}$ such that for any $a\in \FF_{X+1}^*$, there exists $k$ such that $a=a_k$ or $a=a_k^{-1}$. 
 
Now for each $p \in (\FF_{X+1}^*)^n$, let $P_{k}=\{j\in\left[n\right]|\hspace{5pt}p_j=a_k\hspace{5pt}\textrm{or}\hspace{5pt}p_j=a_k^{-1}\}$ for $k=1,2, \ldots, X/2+1$.
Now we claim that
\[
h(p)=|P_1|^2+|P_{2}|(|P_{2}|-1)+ \binom{|P_3|}{2}+\cdots+\binom{|P_{X/2+1}|}{2}
\]
There are three types of contributions to $h(p)$. If $k \notin \{1,2\}$, then $a_k \neq a_k^{-1}$, and each pair $c,d$ of coordinates in $P_k$ causes $p$ to be on $T_{e_c-e_d}$ (if $p_c = p_d$) or on $T_{e_c+e_d}$ (if $p_c=p_d^{-1}$), for a total of $\binom{|P_{i}|}{2}$ hypertori. 
When $k=2$, every pair $c,d \in P_2$ causes $p$ to be on $T_{e_c-e_d}$ \textbf{and} on $T_{e_c+e_d}$, for a total of $|P_{2}|(|P_{2}|-1)$ hyperplanes. 
Finally, when $k=1$, every pair $c,d \in P_1$ causes $p$ to be on two hypertori, and every element $c \in P_1$ causes it to be on $T_{e_c}$, for a total of $|P_{1}|^2$ hyperplanes. 

Moreover, for each partition $\left[n\right]=P_1 \sqcup \cdots \sqcup P_{X/2+1}$ there are $2^{|P_3|+\cdots+|P_{X/2+1}|}$ points $p$ assigned to that partition, because for each $i \in P_k$ with $k \neq 1,2$ we need to choose whether $p_i=a_k$ or $p_i=a_k^{-1}$.
Therefore
\[
\psi_{B_n}(X,Y)=\displaystyle \sum_{[n] = P_1 \sqcup P_2 \sqcup \cdots \sqcup P_{X/2+1}} Y^{|P_1|^2}Y^{|P_{2}|(|P_{2}|-1)}2^{|P_3|}Y^{\binom{|P_3|}{2}} \cdots 2^{|P_{X/2+1}|}Y^{\binom{|P_{X/2+1}|}{2}}
\]
This says that, for each fixed $n$, the coefficients of $Z^n/n!$ (which are polynomial in $X$ and $Y$) in both sides of 
\begin{equation}\label{eq:PsiB}
\Psi_B =\left(\displaystyle \sum_{n\geq 0}Y^{n^2}\dfrac{Z^n}{n!}\right)\left(\displaystyle \sum_{n\geq 0}Y^{n(n-1)}\dfrac{Z^n}{n!}\right)\left(\displaystyle \sum_{n\geq 0}2^n Y^{\binom {n}{2}}\dfrac{Z^n}{n!}\right)^{X/2-1}.
\end{equation}
are equal to each other for infinitely many values of $X$. Therefore they are equal as polynomials, and the above formula holds at the level of formal power series in $X,Y,$ and $Z$, as desired.

\medskip
\noindent 
\textbf{\textsf{Types $\mathsf{C}$ and $\mathsf{D}$}}: We omit these calculations, which are very similar (and slightly easier) than the calculation in type $B$.
\end{proof}

\begin{reptheorem}{thm:TutteR} 
The arithmetic Tutte generating functions of the classical root systems \textbf{in their root lattices} are
\begin{eqnarray*}
\Psi^R_A &=& F(Z,Y)^X\\
\Psi^R_B &=& F(2Z,Y)^{\frac{X}2-1}F(Z,Y^2)F(YZ,Y^2)\\
\Psi^R_C &=& \frac12 F(2Z,Y)^{\frac{X}2-1} \left[F(2Z,Y)+F(YZ,Y^2)^2\right]\\
\Psi^R_D &=& \frac12 F(2Z,Y)^{\frac{X}2-1}\left[F(2Z,Y) + F(Z,Y^2)^2\right]\\
\end{eqnarray*}
\end{reptheorem}

\begin{proof}
The formulas in types $A$ and $B$ are the same as those in Theorem \ref{thm:TutteZ}. The proofs in types $C$ and $D$ are very similar to each other; we carry out the proof for type $C$ explicitly. Let $X+1 \equiv 1 \mod 4$ be a prime.

Consider the lattices
\[
\Lambda_R(C_n) =  \{\varSigma \, a_ie_i \, : \, a_i \in \ZZ, \varSigma a_i \textrm{ is even}\}, \qquad \ZZ^n
\]
and the corresponding tori
\[
T_R = \Hom(\Lambda_R, \FF_{X+1}^*), \qquad 
T_\ZZ = \Hom(\ZZ^n, \FF_{X+1}^*).
\]
We need to compute 
\[
\psi^R_{C_n}(X,Y) = \sum_{f \in T_R}Y^{h(f)}
\]
and to do it we will split the torus $T_R$ into two parts:
\begin{eqnarray*}
T_R^{\textrm{even}} &:=& \{f \in T_R \, : \, f(2e_1) \textrm{ is a square in } \FF_{X+1}^*\}\\
T_R^{\textrm{odd}} &:=& \{f \in T_R \, : \, t(2e_1) \textrm{ is \textbf{not} a square in } \FF_{X+1}^*\}
\end{eqnarray*}
Since $\FF_{X+1}^*$ has as many squares as non-squares, we have $|T_R^{\textrm{even}}| = |T_R^{\textrm{odd}}| = X^n/2$. We will compute the contributions of these two pieces to $\Psi^R_{C_n}$ separately; we call them $\Psi^{R, \textrm{even}}_{C_n}$ and $\Psi^{R, \textrm{odd}}_{C_n}$, respectively.

\begin{lemma}\label{lemma:even} We have
\[
\Psi^{R, \textrm{even}}_{C_n} = \frac12 F(2Z,Y)^{\frac{X}2-1} F(YZ,Y^2)^2.
\]
\end{lemma}

\begin{proof}[Proof of Lemma \ref{lemma:even}]
Since $\Hom(-, G)$ is covariant, the inclusion $i:\Lambda_R \rightarrow \ZZ^n$ gives us a map $i^*:T_\ZZ \rightarrow T_R$. This map is simply given by restriction: it maps $f \in T_\ZZ$ to $i^*f \in T_R$ given by $i^*f(x) = f(i(x)) = f(x)$ for $x \in \Lambda_R$. We claim that $\Im i^* = T_R^{\textrm{even}}$. First notice that if $i^*f \in \Im i^*$ for some $f \in T_\ZZ$ then $i^*f(2e_1) = f(2e_1) = f(e_1)^2$ is a square. In the other direction, assume that $f \in T_R^{\textrm{even}}$ and let $f(2e_1) = x^2$. Define $f_1, f_2: \ZZ^n \rightarrow \FF_{q+1}^*$ by
\begin{eqnarray*}
&f_1(v) =\,  \hat{t}(v) \, \textrm{ and } f_1(v+e_1) = \, x\,  \hat{t}(v) \qquad & \textrm{ for } v \in \Lambda_R. \\
&f_2(v) =\,  \hat{t}(v) \, \textrm{ and } f_2(v+e_1) = \, -x\,  \hat{t}(v) \qquad & \textrm{ for } v \in \Lambda_R. 
\end{eqnarray*}
Clearly $f_1, f_2 \in T_\ZZ$ and $f = i^*f_1 = i^*f_2$. This proves the claim. 

Also, since $|T_R| = |T_\ZZ|= 2|T_R^{\textrm{even}}| = q^n$ and we have constructed two preimages for each element of $\Im i^* = T_R^{\textrm{even}}$, we conclude that $i^*$ is 2-to-1. We also have that $h(\hat{f}) = h(f_1) = h(f_2)$. Therefore
\begin{eqnarray*}
\sum_{f \in T_R^{\textrm{even}}} Y^{h(f)}&=&
\sum_{f_1 \, : \, f \in T_R} Y^{h(f_1)} = \sum_{f_2 \, : \, f \in T_R} Y^{h(f_2)} \\
&=& \frac12 \left(\sum_{f_1 \, : \, f \in T_R} Y^{h(f_1)} + \sum_{f_2 \, : \, f \in T_R} Y^{h(f_2)}\right)  = \frac12 \sum_{f \in T_\ZZ} Y^{h(f)},
\end{eqnarray*}
from which the result follows by Theorem \ref{thm:TutteZ}.
\end{proof}

\begin{lemma}\label{lemma:odd} We have
\[
\Psi^{R, \textrm{odd}}_{C_n} = \frac12 F(2Z,Y)^{\frac{X}2} 
\]
\end{lemma}

\begin{proof}
Let $f \in T_R^{\textrm{odd}}$. Notice that $f(2e_i) = f(2e_1)f(e_i-e_1)^2$ is not a square for any $i$, so we have
\[
\sum_{f \in T_R^{\textrm{odd}}} Y^{h(f)} = \sum_{\stackrel{c_1, \ldots, c_n}{\textrm{ non-squares}}}\,\,  \sum_{f \in T_R^{\textrm{odd}} \, : f(2e_k)=c_k}  Y^{h(f)} 
\]
For each choice of $c_1, \ldots, c_n$ there are $2^{n-1}$ possible choices for $f$ since, for each $i$, $f(e_i-e_{i+1})$ must be one of the two square roots of $f(2e_i)/f(2e_{i+1}) = c_ic_{i+1}^{-1}$, and we can choose freely which one it is. (The product of the two non-squares $c_i$ and $c_{i+1}^{-1}$ is indeed a square.) This determines the remaining values of $f$.

Now notice that $-1$ is a square since $X+1 \equiv 1 \mod 4$, so the $X/2$ non-squares of $\FF_{X+1}^*$ can be split into pairs $\{\alpha_k, \alpha_k^{-1}\}$ for $k=1, \ldots, X/4$. As before, define $P_k = \{i \in [n] \, : \, f(2e_i) = \alpha_k \textrm{ or } \alpha_k^{-1}\}.$ We claim that the inner sum equals
\begin{equation}\label{innereq}
\sum_{f \in T_R^{\textrm{odd}} \, : f(2e_k)=c_k}  Y^{h(f)} = 
2^{X/4-1} \prod_{k=1}^{X/4} \left(\frac12 \sum_{i=0}^{|P_k|} {|P_k| \choose i} Y^{ {i \choose 2} + {|P_k|-i \choose 2}}\right)
\end{equation}
Notice that if $f \in T_{e_i \pm e_j}$, then $f(e_i \pm e_j)=1$ so $f(2e_i)f(2e_j)^{\pm1} = 1$, and $i,j \in P_k$ for some $k$. Therefore each hypertorus containing $f$ is ``contributed" by one of $P_1, \ldots, P_{X/4}$.

Assume  that $P_1 = \{1, \ldots, b\}$ and $f(2e_1) = \cdots = f(2e_b) = c$ for simplicity.\footnote{This choice makes the notation simpler. The reader is invited to see how the argument (very slightly) changes for other choices of $c_1, \ldots, c_n$.}
For $i=1, \ldots, b-1$ we have that $f(e_i-e_{i+1})^2 = f(2e_i)/f(2e_{i+1}) = 1$, so we can choose $f(e_i-e_{i+1})=1$ or $-1$. 
When we have made these $b-1$ decisions, we will be left with one of the $2^{b-1}$ partitions of $P_1$ into two indistinguishable parts $P_1^+$ and $P_1^-$, so that $f(e_i-e_j)=1$ for $i,j$ in the same part, and $f(e_i-e_j)=-1$ for $i,j$ in different parts. This puts $f$ on ${|P_1^+| \choose 2} + {|P_1^-| \choose 2}$ hypertori. (Notice that when $i,j \in P_1$, $f(e_i+e_j) = \pm c \neq 1$ because $c$ is not a square.) This explains the terms $\frac12 \sum
{|P_k| \choose i} Y^{ {i \choose 2} + {|P_k|-i \choose 2}}$ in (\ref{innereq}). However, $f$ is not fully determined yet; for each $1 \leq a \leq \frac{X}4 - 1$
we still have to choose the value of $f(e_i-e_j)$ for some $i \in P_a$ and $j \in P_{a+1}$. There are $2^{X/4-1}$ such choices. This proves (\ref{innereq}).

It remains to observe that for fixed $P_1, \ldots, P_{\frac X 4}$, there are $2^n$ choices of $c_1, \ldots, c_k$. Therefore
\begin{eqnarray*}
\sum_{f \in T_R^{\textrm{odd}}} Y^{h(f)} &=& 
\sum_{[n] = P_1 \sqcup \cdots \sqcup P_{X/4}} 2^{n+ X/4-1} \prod_{k=1}^{X/4} \left(\frac12 \sum_{i=0}^{|P_k|} {|P_k| \choose i} Y^{ {i \choose 2} + {|P_k|-i \choose 2}}\right) \\
&=& 
\frac12 \sum_{[n] = P_1 \sqcup \cdots \sqcup P_{X/4}}  \prod_{k=1}^{X/4} \left(\sum_{i=0}^{|P_k|} 2^{|P_k|} {|P_k| \choose i} Y^{{i \choose 2}+{|P_k|-i \choose 2}}\right)
\end{eqnarray*}
which has exponential generating function
\[
\frac12 \left[ \sum_{n \geq 0} \left(\sum_{i=0}^{n} 2^{n} {n \choose i} Y^{{i \choose 2}+{n-i \choose 2}}\right)\frac{Z^n}{n!}\right]^{X/4} = \frac12 \left(\sum_{n \geq 0}2^nY^{n \choose 2} \frac{Z^n}{n!}\right)^{X/2} = \frac12 F(2Z,Y)^{X/2}
\]
as desired.
\end{proof}
To conclude, we simply combine Lemmas \ref{lemma:even} and \ref{lemma:odd} with the finite field method. Again, some care with ``good primes" is needed: we proceed as in Theorem \ref{thm:TutteZ}.
\end{proof}

\begin{reptheorem}{thm:TutteW} The arithmetic Tutte generating functions of the classical root systems \textbf{in their weight lattices} are
\[
\Psi^W_A = 
\sum_{n \in \NN} \varphi(n)\left(\left[F(Z,Y)F(\omega_n Z, Y) F(\omega_n^2 Z, Y) \cdots F(\omega_n^{n-1} Z, Y)\right]^{X/n}-1 \right)
\]
where $\varphi$ is Euler's totient function and $\omega_n$ is a primitive $n$th root of unity, 
\begin{eqnarray*}
\Psi^W_B &=& F(2Z,Y)^{\frac{X}4-1}F(Z,Y^2)F(YZ,Y^2)\left[F(2Z,Y)^{\frac{X}4} + F(-2Z,Y)^{\frac{X}4}\right]
\\
\Psi^W_C &=& F(2Z,Y)^{\frac{X}2-1}F(YZ,Y^2)^2\\
\Psi^W_D &=& F(2Z,Y)^{\frac{X}4-1}F(Z,Y^2)^2\left[F(2Z,Y)^{\frac{X}4} + F(-2Z,Y)^{\frac{X}4}\right]
\end{eqnarray*}
\end{reptheorem}

%

\begin{proof} 
\noindent 
\textbf{\textsf{Type $\mathsf{A}$}}: We postpone this proof until Section \ref{sec:computebygraphs}.

\noindent 
\textbf{\textsf{Type $\mathsf{B}$}}: 
In this proof we will restrict our attention to primes $X+1$ such that $X$ is a multiple of $4$. Consider the lattices
\[
\ZZ^n = \ZZ\{e_1, \ldots, e_n\}, \qquad 
\Lambda_W =  \ZZ\{e_1, \ldots, e_n, (e_1+\cdots+e_n)/2\} 
\]
and the corresponding tori
$$
T_\ZZ = \Hom(\ZZ^n, \FF_{X+1}^*), \qquad T_W= \Hom(\Lambda_W, \FF_{X+1}^*).
$$
Again, the inclusion $i: \ZZ^n \rightarrow \Lambda_W$ gives a map $i^*: T_W \rightarrow T_{\ZZ}$. Clearly, 
\[
\textrm{Im } i^* = T_W^{\textrm{even}} := \{f \in T_\ZZ \, : \,  f(e_1) f(e_2) \cdots f(e_n) \textrm{ is a square.}\}
\]
and every $f \in  T_W^{\textrm{even}}$ with $f(e_1)\cdots f(e_n) = x^2$ is the image of exactly two maps $f^{\pm1}$, which extend $f$ by defining $f^{\pm 1}(e_1+\cdots+e_n)/2 = \pm x$. Also $h(f)=h(f^+) = h(f^-).$ Therefore
\[
\psi^W_{B_n}(X,Y)= 2 \sum_{p \in T_\ZZ^{\textrm{even}}} Y^{h(p)}
\]
for all primes $X+1$ such that $X$ is a multiple of $m(B)$ for all $B \subseteq B_n$.
%

Now we proceed as in Theorem \ref{thm:TutteZ} for type $B$. Choose $a_1=1, a_2=-1, a_3, \ldots,$ $a_{X/4+1}, a_{X/4+2}, \ldots a_{X/2}$ such that for any $a\in \FF_{X+1}^*$, there exists $k$ such that $a=a_k$ or $a=a_k^{-1}$. Furthermore, since $X$ is a multiple of $4$, $-1$ is a square, and we can assume that $a_1, a_2, \ldots, a_{X/4}+1$ and their inverses are squares, while $a_{X/4}+2, \ldots a_{X/2}$ and their inverses are non-squares.
 
Again, for each $p \in (\FF_{X+1}^*)^n$, let $P_{k}=\{j\in\left[n\right]|\hspace{5pt}p(e_j)=a_k\hspace{5pt}\textrm{or}\hspace{5pt}p(e_j)=a_k^{-1}\}$ for $k=1,2, \ldots, X/2+1$. We still have that 
\[
h(p)=|P_1|^2+|P_{2}|(|P_{2}|-1)+ \binom{|P_3|}{2}+\cdots+\binom{|P_{X/2+1}|}{2}.
\]
Also, for each partition $\left[n\right]=P_1 \sqcup \cdots \sqcup P_{X/2+1}$ there are $2^{|P_3|+\cdots+|P_{X/2+1}|}$ points $p$ assigned to it.

However, we are now interested only in those $p$ such that $p(e_1)\cdots p(e_n)$ is a square. Since the product of non-squares is a square, this holds if and only $|P_{X/4+2}| + \cdots + |P_{X/2}+1|$ is even. Therefore $\psi^W_{B_n}(X,Y) / 2$ equals
\[
\sum_{\stackrel{[n] = P_1 \sqcup P_2 \sqcup \cdots \sqcup P_{X/2+1}}{|P_{X/4+2}| + \cdots + |P_{X/2}+1| \textrm{ is even}.}} 
Y^{|P_1|^2}Y^{|P_{2}|(|P_{2}|-1)}2^{|P_3|}Y^{\binom{|P_3|}{2}} \cdots 2^{|P_{X/2+1}|}Y^{\binom{|P_{X/2+1}|}{2}}.
\]
Now, referring to the formula 
(\ref{eq:PsiB})
for $\Psi_B$ in the proof of Theorem \ref{thm:TutteZ}, we can pick out only the ``even" terms in $\Psi_B$ by means of the following generating function:
\begin{eqnarray*}
\frac{\Psi_B^W}2 &=&
\left(\displaystyle \sum_{n\geq 0}Y^{n^2}\dfrac{Z^n}{n!}\right) \cdot
\left(\displaystyle \sum_{n\geq 0}Y^{n(n-1)}\dfrac{Z^n}{n!}\right) \cdot
\left(\displaystyle \sum_{n\geq 0}2^n\,Y^{\binom {n}{2}}\dfrac{Z^n}{n!}\right)^{X/4-1} \cdot
\\
&& \cdot \,  \frac12 \left[\left(\displaystyle \sum_{n\geq 0}2^n\,Y^{\binom {n}{2}}\dfrac{Z^n}{n!}\right)^{X/4} +
\left(\displaystyle \sum_{n\geq 0}(-1)^n\, 2^n\, Y^{\binom {n}{2}}\dfrac{Z^n}{n!}\right)^{X/4} \right].
\end{eqnarray*}
In the last factor, by introducing minus signs appropriately, we are eliminating the odd terms (which do not contribute to $\Psi_B^W$)  and doubling the even terms (which do contribute). Dividing by $2$, we get the desired formula.

\medskip
\noindent 
\textbf{\textsf{Type $\mathsf{C}$}}: The weight lattice and integer lattice coincide, so $\Psi_C^{W}=\Psi_C$. 

\medskip
\noindent 
\textbf{\textsf{Type $\mathsf{D}$}}: We omit the proof, which is very similar to type $B$ and slightly easier.
\end{proof}

\section{\textsf{Computing arithmetic Tutte polynomials by counting graphs}}
\label{sec:compute2}

In this section we present a different approach towards computing arithmetic Tutte polynomials. 
The key observation is that these computations are closely related to the enumeration of (signed) graphs. To find the arithmetic Tutte polynomials of $A_n, B_n, C_n,$ and $D_n$ with respect to the various lattices of interest, it becomes necessary to count (unsigned / signed) graphs according to (three / six) different parameters. In Section \ref{subsec:graphs} we carry out this enumeration, which may be of independent interest, in Theorems \ref{thm:masterA} and \ref{thm:master}. Then, in Section \ref{subsec:Tutte}, we explain the relationship between signed graphs and classical root systems, and  obtain our main Theorems \ref{thm:Tutte}, \ref{thm:TutteZ}, \ref{thm:TutteR}, and \ref{thm:TutteW} as corollaries.

\subsection{\textsf{Enumeration of graphs and signed graphs}}\label{subsec:graphs}

\subsubsection{\textsf{Enumeration of graphs}}

In this section we consider simple graphs; that is, undirected graphs without loops and parallel edges. Let $g(c,e,v)$ be the number of graphs with $c$ connected components, $e$ edges, and
$v$ vertices labelled $1, \ldots, v$. Let $g'(e,v)$ be the number of \textbf{connected} graphs with $e$ edges and $v$ labelled vertices. Let
\[
G(t,y,z) = \sum_{c,e,v} g(c,e,v) t^c y^e \frac{z^v}{v!}, \qquad
CG(y,z) = \sum_{e,v} g'(e,v) y^e \frac{z^v}{v!}
\]
The main result on graph enumeration that we will need is the following:
\begin{theorem}\label{thm:masterA}
The generating functions for enumerating (connected) graphs are
\[
G(t,y,z) = F(z, 1+y)^t, \qquad CG(y,z) = \log F(z,1+y),
\]
where 
\[
F(\alpha, \beta) = \sum_{n \geq 0} \frac{\alpha^n \, \beta^{n \choose 2}}{n!}
\]
is the deformed exponential function of Definition \ref{def:RR}.
\end{theorem}

The following proof is standard; see for instance \cite[Example 5.2.2]{EC2}. We include it since it is short, and it sets the stage for the more intricate proof of Theorem \ref{thm:master}, the main goal of this section.

\begin{proof}
The compositional formula \cite[Theorem 5.1.4]{EC2} gives
\[
G(t,y,z) = e^{t \, CG(y,z)}
\]
so to compute $G$ it suffices to compute $CG$. In turn, since the above formula gives $CG(y,z) = \log G(1,y,z)$, it suffices to compute $G(1,y,z)$. To do that, observe that there are ${v(v-1)/2 \choose e}$ graphs on $v$ labelled vertices and $e$ edges, so
\[
G(1,y,z) = \sum_{v,e} {{v \choose 2} \choose e}y^e \frac{z^v}{v!} = 
\sum_{v} (1+y)^{v \choose 2} \frac{z^v}{v!} = F(z,1+y).
\]
The desired formulas follow.
 \end{proof}

\subsubsection{\textsf{Enumeration of signed graphs}}

Our goal in this section is to compute the \emph{master generating function} for signed graphs, which enumerates signed graphs according to six parameters. This will be the signed graph analog of Theorem \ref{thm:masterA}.

\begin{definition}  \cite{Z2}
A \emph{signed graph} is a set of vertices, together with a set of positive edges, negative edges, and loops connecting them. A \emph{positive edge} (resp. \emph{negative edge}) is an edge between two vertices, labelled with a $+$ (\emph{resp.} with a $-$). A \emph{loop} connects a vertex to itself; we regard it as a negative edge. There is at most one positive and one negative edge connecting a pair of vertices, and there is at most one loop connecting a vertex to itself. 
\end{definition}

\begin{definition}
A signed graph $G$ is \emph{connected} if and only if its underlying graph $\overline{G}$ (ignoring signs) is connected. The \emph{connected components} of $G$ correspond to those of $\overline{G}$. A \emph{cycle} in $G$ corresponds to a cycle of $\overline{G}$; we call it \emph{balanced} if it contains an even number of negative edges, and \emph{unbalanced} otherwise. We say that $G$ is \emph{balanced} if all its cycles
are balanced.
\end{definition}

\begin{remark}
Note that a loop is an unbalanced cycle, so a signed graph with loops is necessarily unbalanced.
\end{remark}

\begin{definition}
Let $s(c_+,c_-,c_0,l,e,v)$ be the number of signed graphs with 
  $c_+$ balanced components, 
  $c_-$ unbalanced components with no loops, 
  $c_0$ components with loops (which are necessarily unbalanced), 
  $l$ loops, 
  $e$ (non-loop) edges, and
  $v$ vertices.
Let
\[
S(t_+,t_-,t_0,x,y,z) = \sum s(c_+,c_-,c_0,l,e,v)\, t_+^{c_+} \,t_-^{c_-}\, t_0^{c_0} \, x^l y^e \frac{z^v}{v!}
\]
be the \emph{master generating function} for signed graphs.
\end{definition}

\begin{figure}[h]
\begin{center}
\includegraphics[height=1in]{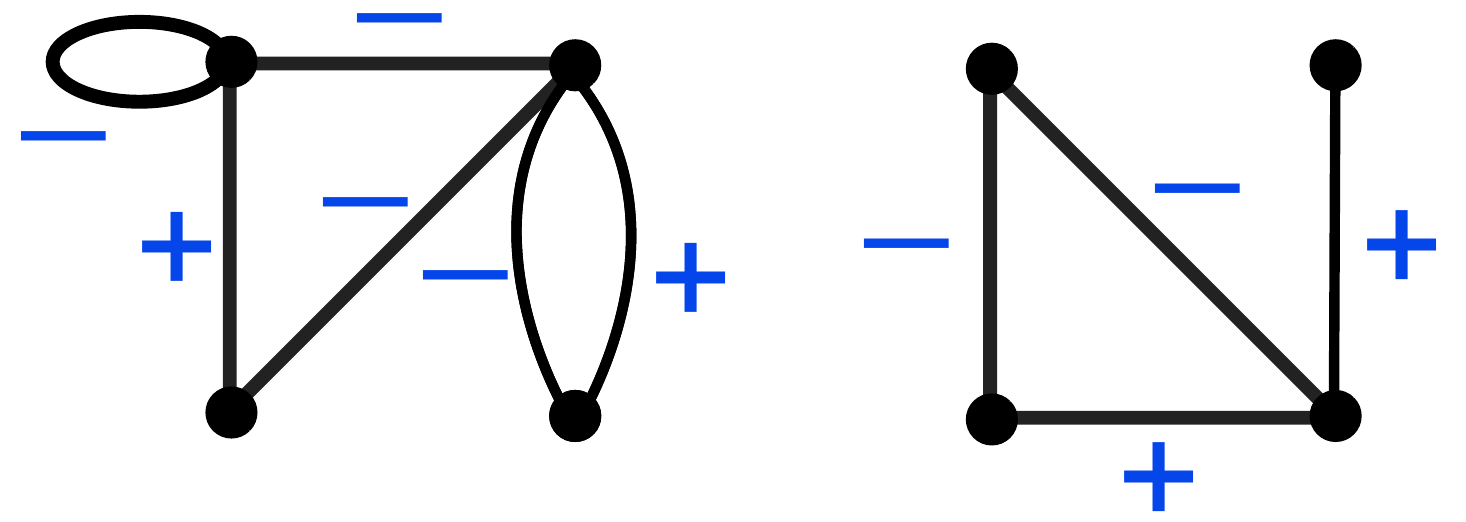} 
\caption{An unbalanced signed graph having one unbalanced and one balanced component.}
\end{center}
\end{figure}

The goal of this section is to prove the following formula.

\begin{theorem}\label{thm:master}
The master generating function for signed graphs is
\[
S(t_+,t_-,t_0,x,y,z) = F(2z, 1+y)^{\frac12(t_+ - t_-)} F(z, (1+y)^2)^{t_--t_0} F((1+x)z, (1+y)^2)^{t_0}
\]
where $F(\alpha, \beta)$ is the deformed exponential function of Definition \ref{def:RR}.
\end{theorem}

For convenience we record below, for each parameter, the letter we use to count it and the formal variable we use to keep track of it in the generating function.

\medskip

\begin{tabular}{|c|c|l|}
\hline
number & variable & parameter \\
\hline
$c_+$ & $t_+$ & balanced components (with no loops) \\
$c_-$ & $t_-$ & unbalanced components with no loops \\
$c_0$ & $t_0$ & (necessarily unbalanced) components with loops \\
$l$ & $x$ & loops\\ 
$e$ & $y$ & edges\\
$v$ & $z$ & vertices\\
\hline
\end{tabular} 
\medskip

In order to compute the master generating function, we will need to count various auxiliary subfamilies of signed graphs, by computing their generating functions. We record them in the following table:

\begin{definition}\label{def:table}
For each type of graph in the right column, let the symbol on the left column denote the number of graphs with the appropriate parameters, and let the entry of the middle column denote the corresponding generating function.
\medskip

\noindent
\begin{tabular}{|c|c|l|}
\hline
number & gen. fn. & type of graph \\
\hline
$s(c_+,c_-,c_0,l,e,v)$ & $S(t_+,t_-,t_0,x,y,z)$ & signed graphs  \\
$r(c_+,c_-,e,v)$ & $R(t_+,t_-,y,z)$ & signed graphs with no loops \\
$b(c_+,e,v)$ & $B(t_+,y,z)$ & balanced signed graphs \\
$s_+(e,v)$ & $CS_+(y,z)$ &  connected balanced signed graphs \\
$s_-(e,v)$ & $CS_-(y,z)$ &  conn. unbal. signed graphs with no loops\\
$s_0(l,e,v)$ & $ CS_0(x,y,z)$ & conn. unbal. signed graphs with loops\\
\hline
\end{tabular}
\medskip

For example, $b(c_+,e,v)$ is the number of balanced signed graphs with $c_+$ (necessarily balanced) components, $e$ edges (which are necessarily non-loops), and $v$ vertices, and
\[
B(t_+,y,z) = \sum b(c_+,e,v) t_+^{c_+} y^e \frac{z^v}{v!}.
\]
\end{definition}
\noindent When it causes no confusion, we will omit the names of the variables in a generating function, for instance, writing $B$ for $B(t_+,y,z)$.

\begin{proof}[Proof of Theorem \ref{thm:master}]
The compositional formula \cite[Theorem 5.1.4]{EC2} gives the following three equations:
\begin{eqnarray*}
B &=&e^{t_+CS_+} \\
R &=& e^{t_+CS_+ + t_-CS_-} \\
S &=& e^{t_+CS_+ + t_-CS_- + t_0CS_0}
\end{eqnarray*}
so we will know $B,R,$ and $S$ if we can compute $CS_+, CS_-$, and $CS_0$. In turn, the above equations give
\begin{eqnarray*}
CS_+(y,z) &=& \frac12 \log  B(2,y,z) \\
CS_+(y,z)+CS_-(y,z) &=& \log R(1,1,y,z) \\
CS_+(y,z)+CS_-(y,z)+CS_0(x,y,z) &=& \log S(1,1,1,x,y,z)
\end{eqnarray*}
and we will see that the right hand sides of these three equations are not difficult to compute. We proceed in reverse order.

First, note that there are ${v \choose l}{v(v-1) \choose e}$ signed graphs on $v$ vertices having $l$ loops and $e$ edges, so we have
\begin{eqnarray*}
S(1,1,1,x,y,z) &=&  \sum_{l,e,v} {v \choose l} {v(v-1) \choose e} x^l y^e \frac{z^v}{v!} \\
&=& \sum_v \left(1+x\right)^v \left(1+y\right)^{2{v \choose 2}} \frac{z^v}{v!} \\
&=& F((1+x)z, (1+y)^2). 
\end{eqnarray*}

Next, observe that there are ${v(v-1) \choose e}$ signed graphs on $v$ vertices having $e$ edges and {\bf no} loops, so we have
\begin{eqnarray*}
R(1,1,y,z) &=&  \sum_{e,v} {v(v-1) \choose e} y^e \frac{z^v}{v!} \\
&=& \sum_v \left(1+y\right)^{2{v \choose 2}} \frac{z^v}{v!} \\
&=& F(z, (1+y)^2) 
\end{eqnarray*}

Finally, counting balanced signed graphs is more subtle. We extend slightly a computation by Kabell and Harary \cite[Correspondence Theorem]{KH}, which relates them to marked graphs. A \emph{marked graph} is a simple undirected graph, together with an assignment of a sign $+$ or $-$ to each \textbf{vertex}. We will enumerate them according to the number of components, edges, and vertices, adding the following entry to the table of Definition \ref{def:table}:

\medskip

\noindent
\begin{tabular}{|c|c|l|}
\hline
number & gen. fn. & type of graph \\
\hline
$\quad\quad m(c_+,e,v) \quad\quad $ & $\quad\quad M(t_+,y,z)\quad\quad$ & marked graphs \quad\quad\quad\quad\quad\quad\quad\quad\\
\hline
\end{tabular}
\medskip

 From each marked graph $G$, we can obtain a signed graph by assigning to each edge of $G$ the product of the signs on its vertices; the resulting graph $G'$ is clearly balanced.

\begin{figure}[h]
\begin{center}
\includegraphics[height=2in]{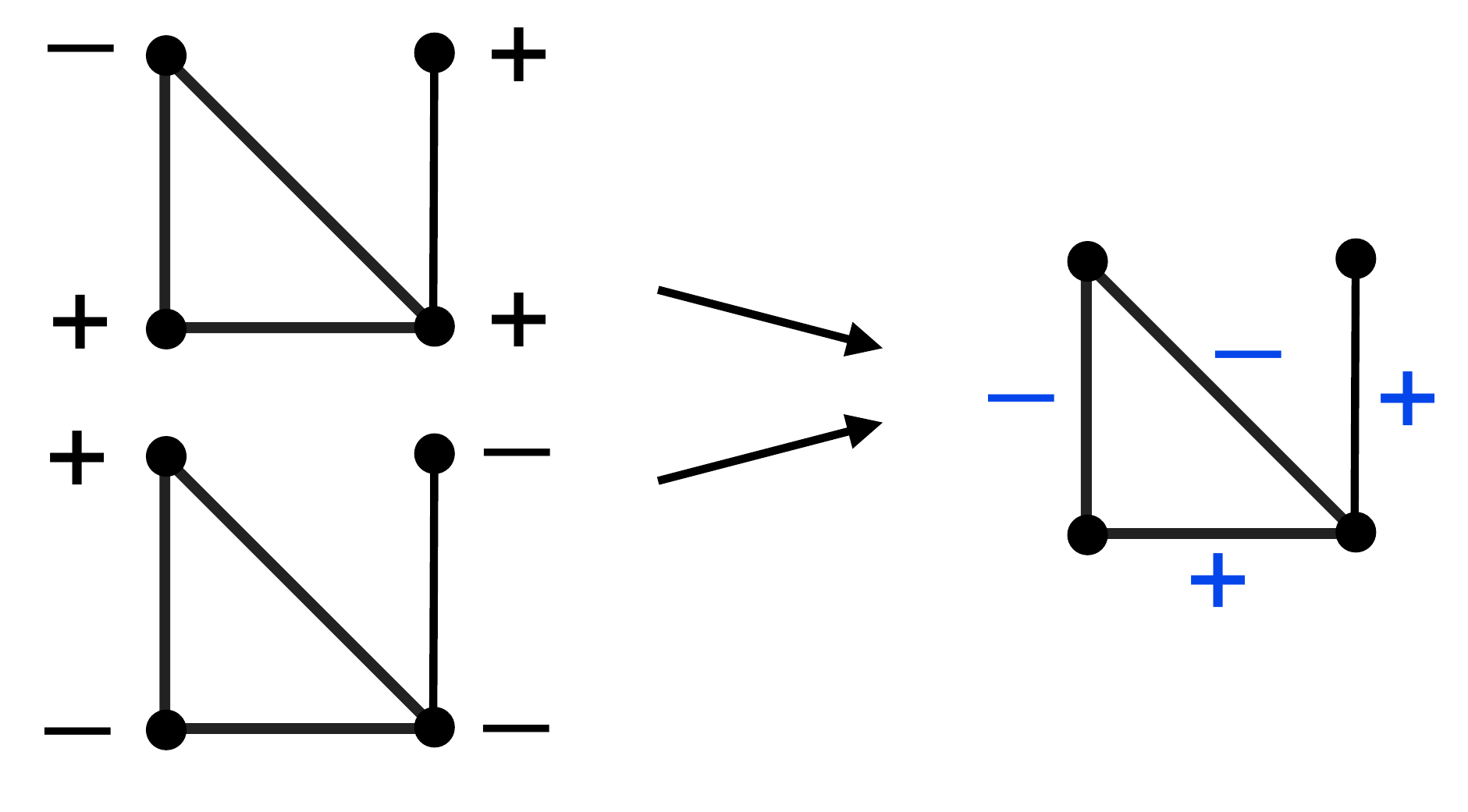} 
\end{center}
\caption{The two marked graphs that give rise to one balanced signed graph.}
\end{figure}

It is not difficult to check that every balanced signed graph $G'$ arises in this way from a graph $G$. Furthermore, if $G'$ has $c_+$ components, then it arises from exactly $2^{c_+}$ marked graphs, obtained from $G$ by choosing some of its components and switching the signs of their vertices. Therefore 
\[
m(c_+,e,v) = 2^{c_+}b(c_+,e,v) \qquad \textrm{and} \qquad M(t_+,y,z) = B(2t_+,y,z).
\]
Now, counting marked graphs is much easier: there are ${v(v-1)/2 \choose e}2^v$ marked graphs on $v$ vertices having $e$ edges. Hence
\[
M(1,y,z) =  \sum_{e, v} {{v(v-1)/2} \choose e} 2^v y^e \frac{z^v}{v!} = \sum_{v \geq 0} \left(1+y\right)^{v \choose 2} \frac{(2z)^v}{v!} = F(2z, 1+y) 
\]
so
\[
B(2,y,z) = F(2z,1+y).
\]

It follows that
\begin{eqnarray*}
CS_+(y,z) &=&  \frac12 \log F(2z, 1+y) \\
CS_-(y,z) &=& \log F(z, (1+y)^2) -  \frac12 \log F(2z, 1+y) \\
CS_0(x,y,z) &=& \log F((1+x)z, (1+y)^2) -  \log F(z, (1+y)^2) 
\end{eqnarray*}
and from this we conclude that
\begin{eqnarray}
B(t_+,y,z) &=& F(2z, 1+y)^{\frac12 t_+} \label{eq:B}\\
R(t_+,t_-,y,z) &=& F(2z, 1+y)^{\frac12(t_+ - t_-)} F(z, (1+y)^2)^{t_-} \notag \\
S(t_+,t_-,t_0,x,y,z) &=& F(2z, 1+y)^{\frac12(t_+ - t_-)} F(z, (1+y)^2)^{t_--t_o} F((1+x)z, (1+y)^2)^{t_0} \notag
\end{eqnarray}
as desired.
\end{proof}

\subsection{\textsf{From classical root systems to signed graphs}}\label{subsec:Tutte}

The theory of signed graphs, developed extensively by Zaslavsky in a series of papers \cite{Z1, Z2, Z3, Z4}, is a very convenient combinatorial model for the classical root systems. To each simple (unsigned) graph $G$ on vertex set $[v]$ we can associate the vector configuration:
\[
A_G = \{e_i - e_j \, : \, ij \textrm{ is an edge of } G, \,\, i<j\} \subseteq A_{v-1}. 
\]
To each signed graph $G$ on $[v]$ we associate the vector configurations\footnote{Zaslavsky follows a different convention, where signed graphs can have \emph{half-edges} and \emph{loops} at a vertex $i$, corresponding to the vectors $2e_i$ and $e_i$ respectively. Since our vector arrangements will never contain $e_i$ and $2e_i$ simultaneously, it will simplify our presentation to consider loops only.}:
\begin{eqnarray*}
B_G &=& 
\{e_i - e_j \, : \, ij \textrm{ is a positive edge of } G, \,\, i<j\} \\
&\cup &
\{e_i + e_j \, : \, ij \textrm{ is a negative edge of } G\} \\
&\cup & 
\{e_i \, : \, i \textrm{ is a loop of } v\}  \subseteq B_v, 
\end{eqnarray*}
\begin{eqnarray*}
C_G &=& 
\{e_i - e_j \, : \, ij \textrm{ is a positive edge of } G, \,\, i<j\} \\
&\cup &
\{e_i + e_j \, : \, ij \textrm{ is a negative edge of } G\} \\
&\cup & 
\{2e_i \, : \, i \textrm{ is a loop of } v\}  \subseteq C_v,
\end{eqnarray*}
and, if $G$ has no loops,
\begin{eqnarray*}
D_G &=& 
\{e_i - e_j \, : \, ij \textrm{ is a positive edge of } G, \,\, i<j\} \\
&\cup &
\{e_i + e_j \, : \, ij \textrm{ is a negative edge of } G\}  \subseteq D_v.
\end{eqnarray*}

These are the subsets of the arrangements $A_v, B_v, C_v$, and $D_v$ (when the signed graph has no loops). Our plan is now to compute the arithmetic Tutte polynomials of these arrangements ``by brute force" directly from the definition:
\[
M_{A}(x,y)= \sum_{B \subseteq {A}} m(B)(x-1)^{r(A)-r(B)}(y-1)^{|B|-r(B)}.
\]
Carrying out such a computation, which is exponential in size, is hopeless for a general arrangement $A$. The structure of these arrangements if very special, however. Here we can use graphs and signed graphs to carry out all the necessary bookkeeping, and their combinatorial properties to compute the desired formulas.

The first step is to the ranks and multiplicities of these vector arrangements in the various lattices. To compute $m(A)=m_\ZZ(A)$ in $\ZZ^V$,  it will be helpful to regard the vectors in $A$ as the columns of a $v \times |A|$ matrix, and recall \cite{MociTutte} that
\[
m(A) =m_\ZZ(A) = \gcd \{ |\det B| \, : \, B \textrm{ full rank submatrix of } A\}.
\]
To compute $m_R(A)$ and $m_W(A)$ will require a bit more care.

Note that when $A$ is $A_G, B_G, C_G$ or $D_G$, the resulting matrix is the adjacency matrix of $G$. Also, dependent subsets of $A_G$ correspond to sets of edges of $G$ containing a cycle, while dependent subsets of $B_G, C_G,$ and $D_G$ correspond to sets of edges containing a balanced cycle. 

\begin{lemma}\label{lemma:A}
$(\mathbf{A_{v-1}})$: 
For any graph $G$ with $v(G)$ vertices and $c(G)$ components,
\[
r(A_G) = v(G) - c(G), \quad m(A_G) = 1
\]
\end{lemma}
\begin{proof}
The first statement is standard \cite{Oxley} and the second follows from the fact that the adjacency matrix of a graph is totally unimodular; \emph{i.e.}, all it subdeterminants equal $1, 0,$ or $-1$. \cite[Chapter 19]{Schrijver}.
\end{proof}

\begin{lemma}\label{lemma:B}
$(\mathbf{B_v})$: 
For any signed graph $G$ with $v(G)$ vertices and $c_-(G)$ unbalanced loopless components,
\[
r(B_G) = v(G) - c_+(G), \qquad m(B_G) = 2^{c_-(G)}
\]
\end{lemma}
\begin{proof}
The first statement is well-known \cite{Z2} and not difficult to prove. For the second one, notice that the matrix $B_G$ can be split into blocks $B_{G_1}, \ldots, B_{G_c}$ where $G_1, \ldots, G_c$ are the connected components of $G$, and all entries outside of these blocks are $0$. Therefore $m(G) = m(G_1) \cdots m(G_c)$. We claim that $m(G_i)$ is $2$ if $G_i$ is unbalanced loopless and $1$ otherwise.

If $G_i$ is balanced, then $B_{G_i}$ has rank $v(G_i)-1$ and the bases of $G_i$ are given by the spanning trees $T$ of $G_i$. Pruning a leaf $v$ of $T$ does not change $|\det A_T|$, as can be seen by expanding by minors in row $v$. Pruning all leaves one at a time, we are left with a single vertex, which has determinant $1$.

If $G_i$ is unbalanced, each basis of $B_{G_i}$ is given by a connected subgraph with a unique unbalanced cycle $C$; and by pruning leaves, the determinant of that basis equals $|\det A_C|$. 

If $G_i$ has a loop, then we can choose $C$ to be that loop, and $|\det A_C|=1$. Therefore 
$m(G_i) = \gcd\{1, \ldots\} = 1$. 

On the other hand, if $G_i$ is loopless, we claim that $|\det A_C| = 2$ for all unbalanced cycles $C$ of $G_i$. To see this, we check that $|\det A_C|$ does not change when we replace two consecutive edges $uv$ and $vw$ by an edge $uw$, whose sign is the product of their signs. At the matrix level, this is an elementary column operation on columns $uv$ and $vw$, followed by an expansion by minors in row $v$. We can do this subsequently until we are left with two edges, which are necessarily of the form $e_i+e_j$ and $e_i-e_j$, so they have determinant $2$. It follows that $m(G_i)=2.$
\end{proof}

\begin{lemma}\label{lemma:C}
$(\mathbf{C_v})$: 
For any signed graph $G$ with $v(G)$ vertices, $c_-(G)$ unbalanced loopless components, and $c_0(G)$ components with loops, 
\[
r(C_G) = v(G) - c_+(G), \qquad m(C_G) = 2^{c_0(G) + c_-(G)}
\]
\end{lemma}
\begin{proof}
The argument used in Lemma \ref{lemma:B} also applies here, but now loops have determinant $2$.
\end{proof}

\begin{lemma}\label{lemma:D}
$(\mathbf{D_v})$: 
For any loopless signed graph $G$ with $v(G)$ vertices and $c_-(G)$ unbalanced loopless components, 
\[
r(C_G) = v(G) - c_+(G), \qquad m(C_G) = 2^{c_-(G)}
\]
\end{lemma}
\begin{proof}
This is a special case of Lemma \ref{lemma:B}.
\end{proof}


To compute the multiplicity functions $m_R$ and $m_W$ with respect to the root and weight lattices requires a bit more care:

\begin{lemma}\label{lemma:A'}
$(\mathbf{A_{v-1}})$: 
For any graph $G$ with $v(G)$ vertices and $c(G)$ components having $v_1, \ldots, v_{c(G)}$ vertices respectively,
\[
m_R(A_G) = 1
\qquad m_W(A_G) = \gcd(v_1, \ldots, v_{c(G)}).
\]
\end{lemma}

\begin{proof}
The first statement is a consequence of Lemma \ref{lemma:A}. For the second one, let $d = \gcd(v_1, \ldots, v_{c(G)})$. We need to show that $[\span A_G \cap \Lambda_W : \ZZ A_G ] = d$.
Consider a vector $a \in \span A_G \cap \Lambda_W$. Recall that $\Lambda_W = \ZZ^n / \1$ where $\1 = (1, \ldots, 1)$, so $a +\lambda \1 \in \ZZ^n$ for some $\lambda \in \RR$. Now, since $a \in \span A_G$, we have $\sum_{v \in G_i} a_v=0$ for each connected component $G_i$ with $1 \leq i \leq c(G)$, which implies that $v_i \lambda = \sum_{v \in G_i}(a_v+\lambda) \in \ZZ$. It follows that $d \lambda \in \ZZ$, so $a \in \ZZ A_G + \frac{k}d \1$ for some $0 \leq k \leq d-1$. Conversely, we see that such an $a$ is in $\span A_G \cap \Lambda_W$.  
The desired result follows.
\end{proof}

\begin{lemma}\label{lemma:B'}
$(\mathbf{B_v})$: 
For any signed graph $G$ with $c_-(G)$ unbalanced loopless components,
\[
m_R(B_G) = 2^{c_-(G)}
\qquad
m_W(B_G) = 
\begin{cases}
2^{c_-(G)} & \mbox{if $G$ has odd balanced components} \\
2^{c_-(G)+1} & \mbox{otherwise}
\end{cases}
\]
\end{lemma}
\begin{proof}The first statement follows from Lemma \ref{lemma:B} since $\Lambda_R = \ZZ^v$. 

For the second one, since $\ZZ B_G \subseteq (\span B_G \cap \Lambda_R) \subseteq (\span B_G \cap \Lambda_W)$ and we already computed $[(\span B_G \cap \Lambda_R) : \ZZ B_G] = 2^{c_-(G)}$, we just need to find the index $
[(\span B_G \cap \Lambda_W) : (\span B_G \cap \Lambda_R)]$ and multiply it by $2^{c_-(G)}$. 

Note that the subspace $\span B_G$ has codimension $c_+(G)$, and is cut out by equations of the form $\pm x_{v_1} \pm \cdots \pm x_{v_c}=0$, one for each balanced component $G_i$ with vertices $v_1, \ldots, v_c$. The signs in this equation depend on the signs of the edges of $G_i$. Now consider $a \in (\span B_G \cap \Lambda_W)$. 

If $G$ has an odd balanced component, then the equation $\pm x_{v_1} \pm \cdots \pm x_{v_c}=0$ corresponding to that component cannot be satisfied by a vector in $\ZZ^v + \frac12 \1$, so $a \in \ZZ^v$. Therefore $(\span B_G \cap \Lambda_W) = (\span B_G \cap \Lambda_R)$.

On the other hand, if $G$ has no odd balanced components, then $a$ can be in $\ZZ^v$ or in $\ZZ^v + \frac12 \1$. Therefore in that case $[(\span B_G \cap \Lambda_W) : (\span B_G \cap \Lambda_R)] = 2$.
\end{proof}

\begin{lemma}\label{lemma:C'}
$(\mathbf{C_v})$: 
For any signed graph $G$ with $c_-(G)$ unbalanced loopless components, and $c_0(G)$ components with loops, 
\[
m_R(C_G) = 
\begin{cases}
1 & \mbox{if $G$ is balanced} \\
2^{c_-(G)+c_0(G)-1} & \mbox{if $G$ is unbalanced}
\end{cases}
\qquad
m_W(C_G) = 2^{c_-(G)+c_0(G)}
\]
\end{lemma}
\begin{proof}
The second statement follows from Lemma \ref{lemma:C} since $\Lambda_W = \ZZ^v$. For the first one, we need to divide 
$2^{c_-(G)+c_0(G)}$ by the index $[(\span B_G \cap \Lambda_W) : \span (B_G \cap \Lambda_R)]$, which we now compute. 
Consider $a \in (\span B_G \cap \Lambda_W)$. 

Recall the equations of $\span B_G$ from the proof of Lemma \ref{lemma:B'}. 
If $G$ is balanced, then all its components are balanced, and combining their equations we get an equation of the form $\pm x_1 \pm \cdots \pm x_v = 0$, which $a$ satisfies. But then $a \in \Lambda_W = \ZZ^v$ implies that $a_1+\cdots+a_v$ is even, which means that  $a \in (\span B_G \cap \Lambda_R)$. It follows that $(\span B_G \cap \Lambda_W) = (\span B_G \cap \Lambda_R)$.

On the other hand, if $G$ is not balanced, then the vertices in the unbalanced components are not involved in the equations of $B_G$. Therefore $a_1+\cdots+a_v$ can be even or odd, and $[(\span B_G \cap \Lambda_W) : (\span B_G \cap \Lambda_R)] = 2$.
\end{proof}

\begin{lemma}\label{lemma:D'}
$(\mathbf{D_v})$: 
For any loopless signed graph $G$ with $c_-(G)$ unbalanced loopless components, 
\begin{eqnarray*}
m_R(C_G) &=& 
\begin{cases}
1 & \mbox{if $G$ is balanced} \\
2^{c_-(G)-1} & \mbox{if $G$ is unbalanced}
\end{cases}
\\
m_W(C_G) &=& 
\begin{cases}
2^{c_-(G)} & \mbox{if $G$ has odd balanced components} \\
2^{c_-(G)+1} & \mbox{otherwise}
\end{cases}
\end{eqnarray*}
\end{lemma}
\begin{proof}
In type $D$, the lattice $\ZZ^v$ is a sublattice of $\Lambda_W$ with index $2$, and contains $\Lambda_R$ as a sublattice of index $2$. Combining the arguments of Lemma \ref{lemma:B'} and \ref{lemma:C'} we obtain the desired result. 
\end{proof}

\subsection{\textsf{Computing the Tutte polynomials by signed graph enumeration}}\label{sec:computebygraphs}

\begin{reptheorem}{thm:Tutte}\cite{Ardila}
The Tutte generating functions of the classical root systems are
\begin{eqnarray*}
\overline{X}_A &=& F(Z,Y)^X\\
\overline{X}_B &=& F(2Z,Y)^{(X-1)/2}F(YZ,Y^2)\\
\overline{X}_C &=& F(2Z,Y)^{(X-1)/2}F(YZ,Y^2)\\
\overline{X}_D &=& F(2Z,Y)^{(X-1)/2}F(Z,Y^2)
\end{eqnarray*}
\end{reptheorem}

Recall from Definition \ref{def:Tuttegen} that the (arithmetic) Tutte generating function is given in terms of the (arithmetic) coboundary polynomials, which are simple transformations of the Tutte polynomial. To write down the exponential generating function $T_\Phi(x,y,z) = \overline{X}_\Phi(X,Y,Z)$ for the actual Tutte polynomials, we simply substitute
\[
X = (x-1)(y-1), \qquad Y=y, \qquad Z=\frac{z}{y-1}.
\]

%

\begin{proof}[Proof of Theorem \ref{thm:Tutte}] In view of Lemmas \ref{lemma:A}, \ref{lemma:B}, and \ref{lemma:C}, each one of these Tutte polynomial computations is a special case of the graph enumeration problems we already solved.

\medskip
\noindent \textsf{\textbf{Type} A:}
By Lemma \ref{lemma:A} we have
\[
T_{A_{v-1}}(x,y)= \sum_{G \textrm{ graphs on } [v]} (x-1)^{c(G)-1}(y-1)^{e(G) - v(G) + c(G)}
\]
for $v \geq 1$, so
\begin{eqnarray*}
1+(x-1)\sum_{v\geq 1}T_{A_{v-1}}(x,y)\frac{z^v}{v!} &=& G\left((x-1)(y-1), y-1, \frac{z}{y-1}\right) \\
&=& G(X,Y-1,Z)
\end{eqnarray*}
where $G$ is the generating function for unsigned graphs, computed in Theorem \ref{thm:masterA}. The result follows.

\medskip
\noindent \textsf{\textbf{Types} B,C:}
 The ordinary Tutte polynomial does not distinguish $B_v$ and $C_v$. Using signed graphs, Lemma \ref{lemma:B} gives
\[
T_{BC_v}(x,y) = \sum_{G \textrm{ signed graphs on } [v]} (x-1)^{c_+(G)}(y-1)^{l(G) + e(G) - v(G) + c_+(G)}
\]
so
\begin{eqnarray*}
\sum_{v\geq 0}T_{BC_v}(x,y)\frac{z^v}{v!} 
&=& S\left((x-1)(y-1),1,1,y-1, y-1, \frac{z}{y-1}\right) \\
&=& S(X,1,1,Y-1,Y-1,Z)
\end{eqnarray*}
where $S$ is the generating function for signed graphs,
and then we can plug in the formula from Theorem \ref{thm:master}. 
 
 \medskip
\noindent \textsf{\textbf{Type} D:}
 Also
\[
T_{D_v}(x,y)= \sum_{\stackrel{G \textrm{ loopless signed}}{\textrm{ graphs on } [v]}} (x-1)^{c_+(G)}(y-1)^{e(G) - v(G) + c_+(G)}.
\]
If we set $t_0=x=0$ in the master generating function for signed graphs, we will obtain the generating function for loopless signed graphs, so
\[
\sum_{v\geq 0}T_{D_v}(x,y)\frac{z^v}{v!} = 
S(X,1,0,0,Y-1,Z)
\]
and the result follows.
\end{proof}

\begin{reptheorem}{thm:TutteZ} The arithmetic Tutte generating functions of the classical root systems \textbf{in their integer lattices} are
\begin{eqnarray*}
\Psi_A &=& F(Z,Y)^X\\
\Psi_B &=& F(2Z,Y)^{\frac{X}2-1}F(Z,Y^2)F(YZ,Y^2)\\
\Psi_C &=& F(2Z,Y)^{\frac{X}2-1}F(YZ,Y^2)^2\\
\Psi_D &=& F(2Z,Y)^{\frac{X}2-1}F(Z,Y^2)^2
\end{eqnarray*}
\end{reptheorem}
%

\begin{proof} We proceed as above.

\medskip
\noindent \textsf{\textbf{Type} A:} In this case $m(A)=1$ for all subsets of $A_{v-1}$, so we still have
\[
1+(x-1)\sum_{v\geq 0}M_{A_{v-1}}(x,y)\frac{z^v}{v!} = 
G(X,Y-1,Z)
\]

\medskip
\noindent \textsf{\textbf{Types} B, C, D:} By Lemmas \ref{lemma:B},
\ref{lemma:C}, and
\ref{lemma:D}, we now have
\begin{eqnarray*}
\sum_{v\geq 0}M_{B_v}(x,y)\frac{z^v}{v!} &=& 
S(X,2,1,Y-1,Y-1,Z) \\
\sum_{v\geq 0}M_{C_v}(x,y)\frac{z^v}{v!} &=& 
S(X,2,2,Y-1,Y-1,Z) \\
\sum_{v\geq 0}M_{D_v}(x,y)\frac{z^v}{v!} &=& 
S(X,2,0,0,Y-1,Z)
\end{eqnarray*}
\noindent  and the desired results follow.
\end{proof}

\begin{reptheorem}{thm:TutteR} The arithmetic Tutte generating functions of the classical root systems \textbf{in their root lattices} are
\begin{eqnarray*}
\Psi^R_A &=& F(Z,Y)^X\\
\Psi^R_B &=& F(2Z,Y)^{\frac{X}2-1}F(Z,Y^2)F(YZ,Y^2)\\
\Psi^R_C &=& \frac12 F(2Z,Y)^{\frac{X}2-1} \left[F(2Z,Y)+F(YZ,Y^2)^2\right]\\
\Psi^R_D &=& \frac12 F(2Z,Y)^{\frac{X}2-1}\left[F(2Z,Y) + F(Z,Y^2)^2\right]\\
\end{eqnarray*}
\end{reptheorem}

%

\begin{proof}
\noindent \textsf{\textbf{Types} A,B:}
Here we have $M_{\Phi}^R = M_{\Phi}$, so the result matches  Theorem \ref{thm:TutteZ}. 

\medskip
\noindent \textsf{\textbf{Types} C, D:}
Lemma \ref{lemma:C'} tells us that
\[
M^R_{C_v} =
\sum_{G \textrm{ unbalanced}} 2^{c_-+c_0-1}(x-1)^{c_+}(y-1)^{l+e-v+c_+} 
+ \sum_{G \textrm{ balanced}} (x-1)^{c_+}(y-1)^{l+e-v+c_+}. 
\]
Since balanced graphs have no loops and satisfy $2^{c_-+c_0}=1$,
\[
2M^R_{C_v} = 
\sum_{G} 2^{c_-+c_0}(x-1)^{c_+}(y-1)^{l+e-v+c_+} 
+\sum_{G \textrm{ balanced}} (x-1)^{c_+}(y-1)^{e-v+c_+}. 
\]
Therefore $\displaystyle 2\sum_{v\geq 0}M^R_{C_v}(x,y)\frac{z^v}{v!}$ equals
\[
S(X,2,2,Y-1,Y-1,Z)+B(X,Y-1,Z)
\]
which gives the desired formula for type $C$, in view of (\ref{eq:B}) in the proof of Theorem \ref{thm:master}. A similar argument works for type $D$.
\end{proof}

\begin{reptheorem}{thm:TutteW} The arithmetic Tutte generating functions of the classical root systems \textbf{in their weight lattices} are
\[
\Psi^W_A = 
\sum_{n \in \NN} \varphi(n)\left(\left[F(Z,Y)F(\omega_n Z, Y) F(\omega_n^2 Z, Y) \cdots F(\omega_n^{n-1} Z, Y)\right]^{X/n}-1\right)
\]
where $\varphi$ is Euler's totient function and $\omega_n$ is a primitive $n$th root of unity, 
\begin{eqnarray*}
\Psi^W_B &=& F(2Z,Y)^{\frac{X}4-1}F(Z,Y^2)F(YZ,Y^2)\left[F(2Z,Y)^{\frac{X}4} + F(-2Z,Y)^{\frac{X}4}\right]
\\
\Psi^W_C &=& F(2Z,Y)^{\frac{X}2-1}F(YZ,Y^2)^2\\
\Psi^W_D &=& F(2Z,Y)^{\frac{X}4-1}F(Z,Y^2)^2\left[F(2Z,Y)^{\frac{X}4} + F(-2Z,Y)^{\frac{X}4}\right]
\end{eqnarray*}
\end{reptheorem}

%

\begin{proof}
\noindent \textsf{\textbf{Type} A:} Let
\[
CG_d(y,z) = \sum_{\stackrel{G \textrm { conn. graph}}{d | v(G)}} y^e \frac{z^v}{v!}, \qquad
G_d(t,y,z)= \sum_{\stackrel{G \textrm{ graph}}{\gcd(v_1, v_2, \ldots v_c)=d}}  t^c y^e \frac{z^v}{v!}
\]
where $v_1, \ldots, v_c$ denote the sizes of the connected components of $G$.
By Lemma \ref{lemma:A'}, $\displaystyle 1+(x-1)\sum_{v\geq 0}M^W_{A_{v-1}}(x,y)\frac{z^v}{v!}=H(X,Y-1,Z)$, where 
\[
H(t,y,z) := \sum_{G \textrm{ graph}} \gcd(v_1, v_2, \ldots v_c) \,  t^c y^e \frac{z^v}{v!} 
= \sum_{d \in \NN} d \, G_d(t,y,z).
\]
Now, the compositional formula gives
\[
e^{tCG_d(y,z)}  = \sum_{\stackrel{G \textrm{ graph}}{d | \gcd(v_1, \ldots, v_c)}}  t^c y^e \frac{z^v}{v!} = 1+\sum_{n \in \NN \, : \, d | n} G_n(t,y,z)
\]
The dual M\"obius inversion formula\footnote{A bit of care is required in applying dual M\"obius inversion, since we are dealing with infinite sums.
We can resolve this by proving the equality separately for each $z$-degree $z^N$, since there are finitely many summands contributing to this degree, namely, those corresponding to divisors of $N$.} \cite[Proposition 3.7.2]{EC1} then gives
\[
G_d(t,y,z) =  \sum_{n \in \NN \, : \, d | n} \mu\left(\frac{n}{d}\right) \left(e^{tCG_n(y,z)}-1 \right) 
\]
where $\mu: \NN \rightarrow \ZZ$ is the M\"obius function. It follows that
\begin{eqnarray}
H(t,y,z) &=& \sum_d \left(d  \sum_{n \in \NN \, : \, d | n}  \mu\left(\frac{n}{d}\right) \left(e^{tCG_n(y,z)}-1\right) \right)  \notag \\
&=& \sum_{n \in \NN} \varphi(n) \, \left(e^{tCG_n(y,z)}-1\right) \label{eq:An}
\end{eqnarray}
where we are using that $\displaystyle \varphi(n) = \sum_{d | n} d \, \mu\left(\frac nd\right)$. \cite{NZ} 

Finally, to write this expression in the desired form, notice that if $\omega$ is a primitive $n$th root of unity, then 
\[
1 + \omega^v + \omega^{2v} + \cdots + \omega^{(n-1)v} = 
\begin{cases} 
n & \mbox{ if } n|v, \\
0 & \mbox{ otherwise.}
\end{cases}
\]
Therefore
\begin{eqnarray*}
CG_n(y,z) &=&  \sum_{\stackrel{G \textrm { conn. graph}}{n | v(G)}} y^e \frac{z^v}{v!} \\
&=& \frac1n \left(CG(y,z)+CG(y, \omega z) + CG(y, \omega^2 z) + \cdots + +CG(y, \omega^{n-1} z) \right) \\
&=&\frac1n  \log\left(F(z,1+y) F(\omega z, 1+y) \cdots F(\omega^{n-1} z, 1+y) \right). 
\end{eqnarray*}
The result follows.

\noindent \textsf{\textbf{Type} B:}
Let a \emph{nobc graph} be a signed graph with no odd balanced components. By Lemma \ref{lemma:B'} we have
\[
M^W_{B_v}(x,y) =
\sum_{G \textrm{ signed}} 2^{c_-}(x-1)^{c_+}(y-1)^{l+e-v+c_+} 
+ \sum_{G \textrm{ nobc}} 2^{c_-}(x-1)^{c_+}(y-1)^{l+e-v+c_+} 
\]
so $\displaystyle \sum_{v\geq 0}M^W_{B_v}(x,y)\frac{z^v}{v!}$ equals
\[
S(X,2,1,Y-1,Y-1,Z) + Q(X,2,1,Y-1,Y-1,Z)
\]
where $Q$ enumerates nobc graphs, using the notation of Definition \ref{def:table}. Let $CQ_+$ enumerate even connected balanced signed graphs.

\medskip
\noindent
\begin{tabular}{|c|c|l|}
\hline
number & gen. fn. & type of graph \\
\hline
$q(c_+,c_-,c_0,l,e,v)$ & $Q(t_+,t_-,t_0,x,y,z)$ & nobc graphs  \\
$q_+(e,v)$ & $CQ_+(y,z)$ &  even conn. bal. signed graphs \qquad\qquad  \\
\hline
\end{tabular}
\medskip

It remains to notice that 
\[
2CQ_+(y,z) = CS_+(y,z) + CS_+(y,-z) 
\]
and by the compositional formula
\[
Q = e^{t_+CQ_+ + t_-CS_-+t_0CS_0}
\]
so
\[
Q= S\left(\frac{t_+}2,t_-,t_0,x,y,z\right) B\left(\frac{t_+}2,y,-z\right).
\]
Plugging in the formulas for $S$ and $B$ we obtain the desired result. 

\noindent \textsf{\textbf{Type} C:}
Here we have $M_{\Phi}^R = M_{\Phi}$ in type $C$.

\noindent \textsf{\textbf{Type} D:}
This case is very similar to type $B$. We omit the details.
\end{proof}

\begin{remark}
As we mentioned in the introduction, our formula for $\Psi_A^W$ seems less tractable than the other ones. However, we have an alternative formulation which is quite efficient for computations.
We can rewrite the expression (\ref{eq:An}) as
\begin{equation}
\Psi_A^W(X,Y,Z) = \sum_{n \in \NN} \varphi(n) \left(e^{X \cdot  CG_n(Y-1,Z)} - 1 \right). \label{eq:An'}
\end{equation}
Notice that the coefficient of $z^N$ in $\Psi_A^W$ only receives contributions from the terms in the right hand side where $n | N$. To compute those contributions, we begin by writing down the first $N$ terms of the generating function $CG(Y-1,Z) = \log F(Z,Y)$. Now computing the  terms of $CG_n(Y-1,Z)$ up to order $z^N$ is trivial: simply keep every $n$-th term of $CG(Y-1,Z)$. Then we plug these into (\ref{eq:An'}) and extract the coefficient of $z^N$.
\end{remark}

\section{\textsf{Arithmetic characteristic polynomials}}\label{sec:1var}

As a corollary, we obtain formulas for the arithmetic characteristic polynomials of the classical root systems, which are given by
\[
\chi^\Lambda_A(q)= (-1)^r q^{n-r} M^\Lambda_A(1-q,0).
\]
For the integer lattices, we have:

\begin{theorem}\label{thm:charZ} The arithmetic characteristic polynomials of the classical root systems \textbf{in their integer lattices} are
\begin{eqnarray*}
\chi^\ZZ_{A_{n-1}}(q) &=& q(q-1)\cdots (q-n+1)\\
\chi^\ZZ_{B_n}(q) &=& (q-2)(q-4)(q-6)\cdots (q-2n+4)(q-2n+2)(q-2n) \\
\chi^\ZZ_{C_n}(q) &=& (q-2)(q-4)(q-6)\cdots (q-2n+4)(q-2n+2) (q-n) \\
\chi^\ZZ_{D_n}(q) &=& (q-2)(q-4)(q-6)\cdots (q-2n+4) (q^2 - 2(n-1)q + n(n-1))
\end{eqnarray*}
\end{theorem}

\begin{proof}
Since $\chi_A(q) = q^{n-r} \psi_A(q,0)$, we can compute the generating function for the characteristic polynomial from the Tutte generating function by substituting $X=q, Y=0, Z=z$. The generating function in type $A$ (where $q^{n-r}=q$) is 
\[
1+ \sum_{n \geq 0} \chi_{A_{n-1}}(q)\frac{z^n}{n!} = 
F(z,0)^q = (1+z)^q = \sum_{n \geq 0} q(q-1)\cdots (q-n+1) \frac{z^n}{n!}.
\]
In type $C$ it is
\[
F(2z,0)^{\frac{q}2-1}F(0,0)^2 = (1+2z)^{\frac{q}2-1}= \sum_{n \geq 0} \left(\frac q2 -1\right)\left(\frac q2 - 2\right)\cdots \left(\frac q2 - n\right) \frac{2^nz^n}{n!} 
\]
from which $\chi_{C_n}(q) = (q-2)(q-4)\cdots (q-2n)$. In type $B$ it is
\[
F(2z,0)^{\frac{q}2-1}F(z,0)F(0,0) = (1+2z)^{\frac{q}2-1}(1+z)
\]
and in type $C$ it is
\[
F(2z,0)^{\frac{q}2-1}F(z,0)^2 = (1+2z)^{\frac{q}2-1}(1+z)^2
\]
from which the formulas follow.
\end{proof}

We obtain similar results for the characteristic polynomials in the other lattices; we omit the details. The most interesting formula that arises is the following:

\begin{reptheorem}{thm:charWA} The arithmetic characteristic polynomials of the root systems $A_{n-1}$ \textbf{in their weight lattices} are given by
\[
\chi^W_{A_{n-1}}(q) =  \frac{n!}q \sum_{m | n} (-1)^{n-\frac nm} \varphi(m) {q/m \choose n/m}
\]
In particular, when $n\geq 3$ is prime, 
\[
\chi^W_{A_{n-1}}(q) = (q-1)(q-2)\cdots (q-n+1) + (n-1)(n-1)!.
\]
\end{reptheorem}

\begin{proof}
Since $F(z,0) = 1+z$, we have
\[
CG(-1,z) = \log(1+z) = -\sum_{k \geq 1}\frac{(-z)^k}{k}
\]
and
\[
CG_m(-1,z) = -\sum_{k \geq 1}\frac{(-z)^{km}}{km} = \frac1m \log(1-(-z)^m).
\]
Substituting $X=q, Y=0, Z=z$ into (\ref{eq:An'}) (where we now have $q^{n-r}=q$) we get
\begin{eqnarray}
1+q\sum_{n \geq 1} \chi^W_{A_{n-1}}(q)\frac{z^n}{n!} &=& \sum_{m \geq 1} \varphi(m) \left(e^{qCG_m(-1,z)}-1\right) \label{sumachi} \\
&=& \sum_{m \geq 1} \varphi(m)\left(\left[1-(-z)^m\right]^{q/m} -1\right) \notag \\
&=& \sum_{m \geq 1} \varphi(m) \left( \sum_{i \geq 1} {q/m \choose i} (-1)^i(-z)^{mi}\right) \notag
\end{eqnarray}
so the coefficient of $z^n$ is 
\[
q \frac{\chi^W_{A_{n-1}}(q)}{n!} = \sum_{m | n} (-1)^{n-\frac nm} \varphi(m) {q/m \choose n/m}
\]
In particular, if $n\geq 3$ is prime we get
\[
q \frac{\chi^W_{A_{n-1}}(q)}{n!} = {q \choose n} + \varphi(n) \frac{q}{n} 
\]
which gives the desired formula.
\end{proof}

A surprising special case is:

\begin{reptheorem}{cor:charW}
If $n,q$ are integers with $n$ odd and $n|q$, then $\chi^W_{A_{n-1}}(q)/n!$ equals the number of cyclic necklaces with $n$ black beads and $q-n$ white beads.
\end{reptheorem}

\begin{proof}
Let $X$ be the set of ``rooted" necklaces consisting of $n$ black beads and $q-n$ white beads around a circle, with one distinguished location called the ``root". There are ${n \choose q}$ of them. The cyclic group $C_q$ acts on $X$ by rotating the necklaces, while keeping the location of the root fixed. Burnside's lemma \cite{Burnside} (which is not Burnside's\cite{Neumann}) then gives us the number of cyclic necklaces: 
\[
\# \textrm{ of $C_q$-orbits of $X$} = \frac1q \sum_{g \in C_q} |X^g|
\]
where $X^g$ is the number of elements of $X$ fixed by the action of $g$. If $(g,q)=d$, then $gX = dX$. 
Let $m=q/d$. A necklace fixed by $g$ consists of $m$ repetitions of the initial string of $d$ beads, starting from the root. Out of these $d=q/m$ beads, $n/m$ must be black. Therefore $|X^g| =  {q/m \choose n/m}$, and we must have $m|q$ and $m|n$.

Now, since there are $\varphi(m)$ elements $g \in C_q$ with $(g,q)=d$, we get
\[
 \# \textrm{ of cyclic necklaces} 
= \frac1q \sum_{m|q, \, m|n} \varphi(m) {q/m \choose n/m}
\]
which matches the expression of Theorem \ref{thm:charWA} when $n$ is odd and $n|q$.
\end{proof}

This result is similar and related to one of Odlyzko and Stanley; see \cite{OdS} and \cite[Problem 1.27]{EC1}. Our proof is  indirect; it would be interesting to give a bijective proof. 


We conclude with a positive combinatorial formula for the coefficients of $\chi^W_{A_{n-1}}(q)$.

\begin{theorem} We have $\displaystyle \chi^W_{A_{n-1}}(q) = \sum_{k=1}^n (-1)^{n-k} c_k q^{k-1}$ for
\[
c_k = \sum_{\pi \in S_{n,k}} gcd(\pi), 
\]
where $S_{n,k}$ is the set of permutations of $n$ with $k$ cycles, and $gcd(\pi)$ denotes the greatest common divisor of the lengths of the cycles of $\pi$.
\end{theorem}

%

\begin{proof}
\noindent
Since 
$$ CG_m(-1,z) = -\sum_{m|k}\frac{(-z)^{k}}{k}$$
the Exponential Formula in its permutation version \cite[Corollary 5.1.9]{EC2} gives 
$$
e^{qCG_m(-1,z)} = 1+ \sum_{n \geq 1} \left((-1)^n\sum_{\pi \in S_n(m)} (-q)^{c(\pi)}\right)\frac {x^n}{n!}
$$
%
where $S_n(m)$ is the set of permutations of $[n]$ whose cycle lengths are multiples of $m$, and $c(\pi)$ denotes the number of cycles of the permutation $\pi$.

Therefore, by (\ref{sumachi}), the coefficient of  $q^k \frac{x^n}{n!}$ for $n\geq 1$ in $1+q\sum_{n \geq 1} \chi^W_{A_{n-1}}(q) {z^n}/{n!}$
is
$$
 \sum_{m \geq 0} \varphi(m) \sum_{\pi \in S_{n,k}(m)} (-1)^{n-k} = 
(-1)^{n-k}  \sum_{\pi \in S_{n,k}} \sum_{m | \gcd(\pi)}  \varphi(m) = 
(-1)^{n-k} \sum_{\pi \in S_{n,k}} \gcd(\pi)
$$
where $S_{n,k}(m)$ is the set of permutations in $S_n(m)$ with $k$ cycles, and we are using that $\sum_{m|d} \varphi(m) = d$.
%
\end{proof}

\section{\textsf{Computations}}\label{sec:computations}

With the aid of Mathematica, we used our formulas to compute the arithmetic Tutte polynomials for the classical root systems in their three lattices in small dimensions. We show the first few polynomials for \textbf{the weight lattice}, which plays the most important role in geometric applications.

%

\bigskip

\medskip

\noindent \begin{tabular}{|l|l|}
\hline
$\Psi$ & Arithmetic Tutte polynomial in the weight lattice \\ \hline
$A_2$				&$1+x$\\		\hline
$A_3$				&$4+x+x^2+3 y$	\\	\hline
$A_4$				&$15+5 x+3 x^2+x^3+20 y+4 x y+12 y^2+4 y^3$\\	\hline
$A_5$				&$96+6 x+11 x^2+6 x^3+x^4+150 y+20 x y+10 x^2 y+135 y^2+15 x y^2$\\
  					&$+95 y^3+5 x y^3+50 y^4+20 y^5+5 y^6 $\\ \hline
$B_2$				&$3+4 x+x^2+4 y+2 y^2$\\ \hline
$B_3$				&$24+17 x+6 x^2+x^3+38 y+10 x y+33 y^2+3 x y^2+22 y^3+12 y^4+6 y^5+2 y^6$\\ \hline
$B_4$				&$153+156 x+62 x^2+12 x^3+x^4+348 y+200 x y+28 x^2 y+438 y^2+132 x y^2$\\
					&$+6 x^2 y^2+420 y^3+60 x y^3+344 y^4+24 x y^4+260 y^5+12 x y^5+184 y^6$\\
					&$+4 x y^6+120 y^7+72 y^8+40 y^9+20 y^10+8 y^11+2 y^12	$\\ \hline
$B_5$				&$1680+1409 x+580 x^2+150 x^3+20 x^4+x^5+4604 y+2436 x y+580 x^2 y$\\
					&$+60 x^3 y+6910 y^2+2350 x y^2+330 x^2 y^2+10 x^3 y^2+7830y^3+1780 x y^3$\\
					&$+150 x^2 y^3+7620 y^4+1200 x y^4+60 x^2 y^4+6846 y^5+804 x y^5+30 x^2 y^5$\\
					&$+5844 y^6+506 x y^6+10 x^2 y^6+4780 y^7+300 x y^7+3780 y^8+180x y^8$\\
					&$+2900 y^9+100 x y^9+2154 y^{10}+50 x y^{10}+1540 y^{11}+20 x y^{11}+1055y^{12}$\\
					&$+5 x y^{12}+690 y^{13}+430 y^{14}+254 y^{15}+140 y^{16}+70 y^{17}+30$\\ \hline
$C_2$				&$3+4 x+x^2+4 y+2 y^2$	\\	\hline
$C_3$				&$15+23 x+9 x^2+x^3+32 y+16 x y+30 y^2+6 x y^2+20 y^3+12 y^4+6 y^5+2 y^6$	\\ \hline
$C_4$				&$105+176 x+86 x^2+16 x^3+x^4+296 y+240 x y+40 x^2 y+396 y^2+168 x y^2$\\
					&$+12 x^2 y^2+376 y^3+88 x y^3+304 y^4+48 x y^4+232 y^5+24 x y^5+168 y^6+8 x y^6$\\
					&$+112 y^7+70 y^8+40 y^9+20 y^{10}+8 y^{11}+2 y^{12}$	\\ \hline
$C_5$				&$945+1689 x+950 x^2+230 x^3+25 x^4+x^5+3264 y+3376 x y+960 x^2 y+80 x^3 y$\\
					&$+5540 y^2+3500 x y^2+540 x^2 y^2+20 x^3 y^2+6640 y^3+2720 x y^3+240 x^2 y^3$\\
					&$+6600 y^4+1920 x y^4+120 x^2 y^4+5956 y^5+1344 x y^5+60 x^2 y^5+5084 y^6$\\
					&$+896 x y^6+20 x^2 y^6+4160 y^7+560 x y^7+3310 y^8+350 x y^8+2580 y^9+200 x y^9$\\
					&$+1952 y^{10}+100 x y^{10}+1420 y^{11}+40 x y^{11}+990 y^{12}+10 x y^{12}+660 y^{13}+420 y^{14}$\\
					&$+252 y^15+140 y^16+70 y^17+30 y^18+10 y^19+2 y^20$\\ \hline
$D_2$				&$1+2x+x^2$\\ \hline
$D_3$				&$15+5 x+3 x^2+x^3+20 y+4 x y+12 y^2+4 y^3$\\ \hline
$D_4$				&$57+88 x+38 x^2+8 x^3+x^4+160 y+112 x y+16 x^2 y+216 y^2+72 x y^2+200 y^3$\\
					&$+24 x y^3+140 y^4+80 y^5+40 y^6+16 y^7+4 y^8$ \\ \hline
$D_5$				&$915+629 x+270 x^2+90 x^3+15 x^4+x^5+2384 y+1096 x y+320 x^2 y+40 x^3 y$\\
					&$+3540 y^2+1080 x y^2+180 x^2 y^2+4060 y^3+840 x y^3+60 x^2 y^3+3930 y^4$\\
					&$+510 x y^4+3376 y^5+264 x y^5+2644 y^6+116 x y^6+1920 y^7+40 x y^7+1310 y^8$\\
					&$+10 x y^8+840 y^9+504 y^{10}+280 y^{11}+140 y^{12}+60 y^{13}+20 y^{14}+4 y^{15}$ \\
\hline
\end{tabular}

\newpage

This also gives us the arithmetic characteristic polynomial of $A$ and the Ehrhart polynomial of the zonotope $Z(A)$:
\[
\chi_A(q)=(-1)^r M_A(1-q,0), \qquad E_{Z(A)}(t) = t^rM_A(1+1/t,1)
\]
As explained in the introduction, the arithmetic characteristic polynomial carries much information about the complement of the toric arrangement of $A$, over the compact torus $(\SS_1)^n$, the complex torus $(\CC^*)^n$, or the finite torus $(\FF_q^*)^n$. When $A$ is a finite root system, Moci showed \cite{Mociroot} that $|\chi_A(0)|$ equals the size of the Weyl group $W$, as can be verified below.

The Ehrhart polynomial enumerates the lattice points of the zonotope $Z(A)$. As explained in the introduction, it also gives us the dimensions of the Dahmen--Micchelli space and the De Concini--Procesi--Vergne space.

\medskip

\noindent \begin{tabular}{|l|l|l|}
\hline
$\Psi$ & Characteristic Polynomial & Ehrhart Polynomial \\ \hline
$A_2$	&$-2+q$											& $1+2t$	\\
$A_3$	&$6 - 3 q + q^2$									& $1+3 t+9 t^2$\\
$A_4$	&$-24 + 14 q - 6 q^2 + q^3$						& $1+6 t+18 t^2+64 t^3$\\
$A_5$	&$120 - 50 q + 35 q^2 - 10 q^3 + q^4$			& $1+10 t+45 t^2+110 t^3+625 t^4$\\
\hline
$B_2$	&$8-6 q+q^2$										&$1+6 t+14 t^2$\\
$B_3$	&$-48+32 q-9 q^2+q^3$							&$1+9 t+45 t^2+174 t^3$\\
$B_4$	&$384-320 q+104 q^2-16 q^3+q^4$					&$1+16 t+138 t^2+820 t^3+3106 t^4$\\
\multirow{2}{*}{$B_5$}	&$-3840+3104 q-1160 q^2+240 q^3$	&$1+25 t+310 t^2+2530 t^3+15365 t^4$\\
 		&$\qquad-25 q^4+q^5$									&$\qquad+72290 t^5$\\
\hline
$C_2$	&$8-6 q+q^2$										&$1+6 t+14 t^2$\\
$C_3$	&$-48+44 q-12 q^2+q^3$							&$1+12 t+66 t^2+172 t^3$\\
$C_4$	&$384-400 q+140 q^2-20 q^3+q^4$					&$1+20 t+192 t^2+1080 t^3+3036 t^4$\\
\multirow{2}{*}{$C_5$}	&$-3840+4384 q-1800 q^2+340 q^3$	&$1+30 t+440 t^2+4040 t^3+23580 t^4$\\
		&$\quad-30 q^4+q^5$									&$\quad +69976 t^5$\\
\hline
$D_2$	&$4-4 q+q^2$										&$1+4 t+4 t^2$\\
$D_3$	&$-24+14 q-6 q^2+q^3$							&$1+6 t+18 t^2+64 t^3$\\
$D_4$	&$192-192 q+68 q^2-12 q^3+q^4$					&$1+12 t+84 t^2+432 t^3+1272 t^4$\\
\multirow{2}{*}{$D_5$}	&$-1920+1504 q-640 q^2+160 q^3$	&$1+20 t+200 t^2+1320 t^3+6700 t^4$\\
		&$\quad-20 q^4+q^5$									&$\quad+31488 t^5$\\
\hline
\end{tabular}

\medskip

Recall that the leading coefficient of the Ehrhart polynomial $E_{Z(A)}(t)$ is the lattice volume of the zonotope $Z(A)$. \cite{EC1} In type $A$ it is well known \cite{Stanleyzonotope} that the $\ZZ^n$ lattice volume of $Z(A_n)$ is $n^{n-2}$, so its weight lattice volume is $n^{n-1}$, as the examples show. 

%
%
%

\section{\textsf{Acknowledgments}}
The authors would like to thank Petter Br\"anden, Luca Moci, and Monica Vazirani for enlightening discussions on this subject.
Part of this work was completed while the first author was on sabbatical leave at the University of California, Berkeley; he would like to thank Lauren Williams, Bernd Sturmfels, and the mathematics department for their hospitality.

\bibliographystyle{plain}

\small


\begin{thebibliography}{99}

\bibitem{Ardila}
F. Ardila. Computing the Tutte polynomial of a hyperplane arrangement. \textit{PaciÞc J. Math.}
\textbf{230} (2007) 1-17.


\bibitem{Athanasiadis}
C. A. Athanasiadis. Characteristic polynomials of subspace arrangements and Þnite Þelds, \textit{Adv. Math.} \textbf{122} (1996), 193-233.


\bibitem{Bj05}
A. Bj\"{o}rner and F. Brenti. \emph{Combinatorics of
Coxeter groups}, Springer-Verlag, New York, 2005.


\bibitem{BM}
Petter Br\"and\'en and Luca Moci. \emph{The multivariate arithmetic Tutte polynomial}. To appear, Transactions of the AMS, 2013.

\bibitem{Burnside}
W. Burnside. Theory of groups of finite order. Cambridge University Press, 1897.


\bibitem{Mocigeom}
F. Cavazzani and L. Moci. Geometric realizations and duality for Dahmen-Micchelli modules and De Concini-Procesi-Vergne modules. arXiv:1303.0902. Preprint, 2013. 

\bibitem{CrapoRota}
H. Crapo and G.-C. Rota. \textit{On the foundations of combinatorial theory: combinatorial geometries}, MIT Press, Cambridge, MA,
1970.

\bibitem{DM1}
W. Dahmen and C. A. Micchelli. \emph{On the solution of certain systems of partial difference equations and linear dependence of translates of box splines.} Trans. Amer. Math. Soc. {\bf 292} (1985) 305-320.

\bibitem{DM2} 
W. Dahmen and C. A. Micchelli. \emph{The number of solutions to linear Diophantine equations and multivariate splines} Trans. Amer. Math. Soc. {\bf  308} (1988) 509-532.


\bibitem{DPtoric}
C. De Concini, and C. Procesi. \emph{On the geometry of toric arrangements} Transformation Groups (2005) \textbf{10} 387-422.

\bibitem{DPzonotope}
C. De Concini and C. Procesi. \emph{The zonotope of a root system.} Transformation Groups (2008) \textbf{13} 507-526.

\bibitem{DPV1}
C. De Concini, C. Procesi.  M. Vergne. \emph{Vector partition functions and index of transversally elliptic operators.} Transformation Groups (2010) 775-811.

\bibitem{DPV2}
C. De Concini, C. Procesi.  M. Vergne. \emph{Vector partition function and generalized Dahmen-Micchelli spaces.} Transformation Groups (2010) 751-773.

\bibitem{Dirichlet}
P.G.L. Dirichlet. \emph{Beweis des Satzes, dass jede unbegrenzte arithmetische Progression, deren erstes Glied und Differenz ganze Zahlen ohne gemeinschaftlichen Factor sind, unendlich viele Primzahlen enthŠlt}. Abhand. Ak. Wiss., (1837) 45--81.




\bibitem{ERS} R. Ehrenborg, M. Readdy, M. Sloane Affine and Toric Hyperplane Arrangements
Discrete and Computational Geometry
June 2009, Volume 41, Issue 4, pp 481-512


\bibitem{FH}
W. Fulton, J. Harris. \textit{Representation Theory: A First Course} Springer-Verlang
, New York, 1991

\bibitem{Geldon}
T. Geldon. Computing the Tutte polynomial of hyperplane arrangements. Ph.D. Thesis, U. of Texas, Austin, 2009.

\bibitem{GM}
Mark Goresky and Robert MacPherson, Stratified Morse theory, Springer-Verlag, 1988.



\bibitem{KH}
F. Harary and J.A. Kabell. \emph{Counting balanced signed graphs using marked graphs.} Proc. Edinburgh Math. Soc. {\bf 24} (1981) 99-104.

\bibitem{Hu90}
J. Humphreys. \emph{Reflection groups and Coxeter groups},
Cambridge Studies in Advanced Mathematics 29, Cambridge University
Press, Cambridge, 1990.


\bibitem{Lang}
S. Lang. \textit{Algebra, Revised Third Edition} Springer-Verlag, New York, 2002

\bibitem{Langley}
J.K. Langley, \emph{A certain functional-differential equation}, J. Math. Anal. Appl. 244, 564Ð567 (2000).

\bibitem{Liu}
Y. Liu, \emph{On some conjectures by Morris et al. about zeros of an entire function}, J. Math. Anal. Appl. 226, 1Ð5 (1998).

\bibitem{Mociroot}
L. Moci. \emph{Combinatorics and topology of toric arrangements defined by root systems}. Rend. Lincei Mat. e Appl. {\bf 19} (2008), 293-308.



\bibitem{MociTutte}
L. Moci. \emph{A Tutte polynomial for toric arrangements.} Transactions Amer. Math. Soc. {\bf 364} (2012), 1067-1088.


\bibitem{MociEhrhart}
L. Moci and M. D'Adderio. \emph{Ehrhart polynomial and arithmetic Tutte polynomial} European J. of Combinatorics {\bf 33} (2012) 1479Ð1483.
 
 
\bibitem{Morris}
G.R. Morris, A. Feldstein and E.W. Bowen, The Phragm\'en -- Lindel\"of principle and a class of functional differential equations, in Ordinary Differential Equations: 1971 NRL-MRC Conference, edited by L. Weiss (Academic Press, New York, 1972), pp. 513--540.

\bibitem{Neumann}
P. Neumann.  \emph{A lemma that is not Burnside's} The Mathematical Scientist {\bf 4} (1979) 133Ð141.

\bibitem{NZ}
I. Niven, H. Zuckerman, H. Montgomery. An Introduction to the Theory of Numbers. Wiley Publishers, 5th edition, 1991.

\bibitem{OdS}
A. Odlyzko and R. Stanley. \emph{Enumeration of power sums modulo a prime} J. Number Theory {\bf 10} (1978), 263-272.


\bibitem{OS}
P. Orlik and L. Solomon. Combinatorics and topology of complements of
hyperplanes, Invent. Math. \textbf{56} (1980), 167-189.


\bibitem{Or92}
P. Orlik and H. Terao. \emph{Arrangements of hyperplanes},
Springer-Verlag, Berlin/Heidelberg/New York, 1992.


\bibitem{Oxley}
J. Oxley. Matroid Theory, Second edition, Oxford University Press, New York, 2011.

\bibitem{PS}
A. Postnikov and R. Stanley. \emph{Deformations of Coxeter hyperplane arrangements} J. Combinatorial Theory (A) {\bf 91} (2000), 544-597. 


\bibitem{Schrijver}
A. Schrijver. Theory of linear and integer programming. Wiley Publishers, 1998.

\bibitem{SS}
A. Scott and A. Sokal. \emph{Some variants of the exponential formula, with application to the multivariate Tutte polynomial (alias Potts model).} SŽminaire Lotharingien de Combinatoire 61A (2009), Article B61Ae.

\bibitem{SS2}
A. Scott and A. Sokal. \emph{The repulsive lattice gas, the independent-set polynomial, and the Lov‡sz local lemma.}
J. Stat. Phys.  {\bf 118} (2005) 1151Ð1261. 

\bibitem{Sokal}
A. Sokal. \emph{The leading root of the partial theta function}. Advances in Mathematics \textbf{229} (2012) 2603--2621.

\bibitem{Stanleyzonotope}
Richard P. Stanley. \emph{A zonotope associated with graphical degree sequences.} Applied geometry and discrete mathematics, Volume 4. DIMACS Ser. Discrete Math. Theoret. Comput. Sci. 555Ð570. Amer. Math. Soc., Providence, RI, 1991.

\bibitem{EC1}
R. P. Stanley. \textit{Enumerative Combinatorics, vol. 1} Cambridge University
Press, Cambridge, 1999.

\bibitem{EC2}
R. P. Stanley. \textit{Enumerative Combinatorics, vol. 2} Cambridge University
Press, Cambridge, 1999.

\bibitem{Sthyps}
R. P. Stanley. \emph{Lectures on hyperplane arrangements.} Geometric Combinatorics (E. Miller, V. Reiner, and B. Sturmfels, eds.), IAS/Park City Mathematics Series, vol. 13, American Mathematical Society, Providence, RI, 2007, pp. 389-496.
 
\bibitem{Tutte}
W.T. Tutte. \emph{On dichromatic polynominals}. J. Combin. Theory {\bf 2} (1967)  301Ð320.


\bibitem{Welsh}
Dominic Welsh. Complexity: Knots, Colourings and Counting. London Mathematical Society Lecture Note Series. Cambridge University Press.

\bibitem{WW}
Dominic J.A. Welsh, Geoffrey P. Whittle. \emph{Arrangements, Channel Assignments, and Associated Polynomials}. Advances in Applied Mathematics, Volume 23, Issue 4, November 1999, Pages 375-406.

\bibitem{Zaslavsky}
T. Zaslavsky. Facing up to arrangements: face-count formulas for partitions
of space by hyperplanes, \textit{Mem. Amer. Math. Soc.}, vol. 1, no. 154, 1975.


\bibitem{Z1}
T. Zaslavsky. \emph{The geometry of root systems and signed graphs}. Amer. Math. Monthly {\bf 88} (1981), 88-105.

\bibitem{Z2}
T. Zaslavsky.
\emph{Signed graphs}. Discrete Appl. Math., {\bf 4} (1982) 47-74.

\bibitem{Z3}
T. Zaslavsky.
\emph{A mathematical bibliography of signed and gain graphs and allied areas}. Electronic J. Combin., Dynamic Surveys in Combinatorics (2012), No. DS8. 

\bibitem{Z4}
T. Zaslavsky.
\emph{Signed graphs and geometry}. Int. Workshop on Set-Valuations, Signed Graphs, Geometry and Their Applications.  J. Combin. Inform. System Sci. {\bf 37} (2012), 95--143. 


\end{thebibliography}

\end{document}